\documentclass[a4paper,11pt]{amsart}
\usepackage[utf8]{inputenx}
\usepackage[T1]{fontenc}
\usepackage{amssymb,amsmath,amsfonts,amsthm}
\usepackage{mathrsfs}
\usepackage[all]{xy}
\usepackage{multirow}
\usepackage{fullpage}
\usepackage[colorlinks=true,linktoc=page]{hyperref}

\makeatletter
\@namedef{subjclassname@2020}{\textup{2020} Mathematics Subject Classification}
\makeatother

\DeclareMathOperator{\Spec}{Spec}
\DeclareMathOperator{\ima}{im}

\DeclareMathOperator{\MW}{MW}
\DeclareMathOperator{\W}{W}
\DeclareMathOperator{\odd}{odd}

\DeclareMathOperator{\Qlc}{Qlc}
\DeclareMathOperator{\Qld}{Qld}
\DeclareMathOperator{\AQld}{AQld}

\DeclareMathOperator{\Chow}{CH}
\DeclareMathOperator{\di}{div}

\newcommand{\un}{\underline}
\newcommand{\disp}{\displaystyle}
\newcommand{\fun}[4]{\left\{\begin{matrix} #1 & \to & #2 \\ #3 & \mapsto & #4 \end{matrix}\right.}
\newcommand{\C}{\mathbb{C}}
\newcommand{\R}{\mathbb{R}}
\newcommand{\Q}{\mathbb{Q}}
\newcommand{\Z}{\mathbb{Z}}
\newcommand{\N}{\mathbb{N}}
\newcommand{\Sm}{\mathbb{S}}
\newcommand{\Os}{\mathcal{O}}
\newcommand{\I}{\mathscr{I}}
\newcommand{\A}{\mathbb{A}}

\newcommand{\As}[2]{\mathbb{A}^{#1}_{#2} \setminus \{0\}}

\newcommand{\normalsheaf}[2]{\mathcal{N}_{#1/#2}}
\newcommand{\cono}[2]{\mathcal{C}_{#1/#2}}
\newcommand{\Lb}{\mathcal{L}}
\newcommand{\Vb}{\mathcal{V}}
\newcommand{\Cm}{\mathcal{C}}

\newcommand{\KS}[3]{K^{\MW}_{#2}(#1,#3)}
\newcommand{\K}[2]{K^{\MW}_{#2}(#1)}
\newcommand{\KMWF}[2]{\un{K}^{\MW}_{#1}\{#2\}}
\newcommand{\KF}[1]{\un{K}^{\MW}_{#1}}
\newcommand{\CHMW}[4]{H^{#2}(#1,\KMWF{#3}{#4})}
\newcommand{\CH}[3]{H^{#2}(#1,\KF{#3})}
\newcommand{\qlc}[1]{\Qlc_{#1}}
\newcommand{\qld}[1]{\Qld_{#1}}
\newcommand{\aqld}[1]{\AQld_{#1}}
\newcommand{\seifert}{\mathcal{S}}

\newtheorem{theorem}{Theorem}[section]
\newtheorem{theoremdefinition}[theorem]{Theorem-Definition}
\newtheorem{proposition}[theorem]{Proposition}
\newtheorem{lemma}[theorem]{Lemma}
\newtheorem{lemmadefinition}[theorem]{Lemma-Definition}
\newtheorem{corollary}[theorem]{Corollary}

\theoremstyle{definition}

\newtheorem{definition}[theorem]{Definition}
\newtheorem{notation}[theorem]{Notation}
\newtheorem{example}[theorem]{Example}
\newtheorem{examples}[theorem]{Examples}

\theoremstyle{remark}

\newtheorem{remark}[theorem]{Remark}

\begin{document}

\title{The quadratic linking degree}
\author{Cl\'ementine Lemari\'e\texttt{-{-}}Rieusset}
\address{Institut de Math\'ematiques de Bourgogne, UMR 5584, CNRS, Universit\'e de Bourgogne, F-21000 Dijon, France}
\address{Fakultät für Mathematik, Universität Duisburg-Essen, Campus Essen, 45117 Essen, Germany}
\email{clementine.lemarierieusset@uni-due.de}

\begin{abstract}
	By using motivic homotopy theory, we introduce a counterpart in algebraic geometry to oriented links and their linking numbers. After constructing the (ambient) quadratic linking degree --- our analogue of the linking number which takes values in the Witt group of the ground field --- and exploring some of its properties, we give a method to explicitly compute it. We illustrate this method on a family of examples which are analogues of torus links, in particular of the Hopf and Solomon links.
\end{abstract}



\maketitle

\tableofcontents

\section{Introduction}\label{secintro}

In 1999, Morel and Voevodsky founded motivic homotopy theory (see \cite{morelvoevodsky}) in order to import topological methods into algebraic geometry (see also the second section of \cite{AsokOst} for a brief overview). The goal of this paper is to explore the possibility of defining a counterpart to knot theory in algebraic geometry by using motivic homotopy theory. Specifically, we develop an algebro-geometric theory of linking in a setting inspired by classical linking in knot theory and follow it through to explicit computations on families of algebraic varieties.

We begin by defining counterparts, over a perfect field $F$, to oriented links with two components (i.e. couples of disjoint oriented knots). Specifically, we replace the circle $\Sm^1$ with $\As{2}{F}$ and the $3$-sphere $\Sm^3$ with $\As{4}{F}$. We then define a counterpart to the linking number, which in knot theory is an invariant of oriented links with two components which corresponds to the number of times one of the oriented components turns around the other oriented component, and a counterpart to the linking couple, which is a couple of integers whose absolute values coincide with the absolute value of the linking number. We call these counterparts the ambient quadratic linking degree and the quadratic linking degree (couple) respectively. To define these, we use Chow-Witt groups, and more generally Rost-Schmid groups, instead of the singular cohomology groups used in classical knot theory. The Rost-Schmid groups are the cohomology groups of the Rost-Schmid complex (see Appendix \ref{secprelim}) which is built from Milnor-Witt $K$-theory groups; this is where motivic homotopy theory comes into play, as Milnor-Witt $K$-theory groups are in fact groups of morphisms of motivic spheres in the stable motivic homotopy category (see \cite[Corollary 1.25]{morel}).

In knot theory, the cohomological definitions of the linking number use the notion of Seifert class of an oriented knot, which is the class (in cohomology) of Seifert surfaces of an oriented knot (compact connected oriented surfaces whose oriented boundary is the oriented knot; see for instance \cite[Chapter 5, Sections A and D]{rolfsen}). The following theorem-definition establishes the possibility of defining the quadratic linking class, the quadratic linking degree and the ambient quadratic linking degree in a similar manner to the linking class, the linking couple and the linking number respectively (see \cite[Chapter 1]{myphdthesis} for more details on these).


\begin{theoremdefinition}[Quadratic linking degrees]
	Let $Z = Z_1 \sqcup Z_2 \subset \As{4}{F}$, with $Z_1 \simeq \As{2}{F}$ and $Z_2 \simeq \As{2}{F}$, be an oriented link with two components (see Definition \ref{deflink} for details). There exist two elements of the Chow-Witt group $\widetilde{CH}^1((\As{4}{F}) \setminus Z)$ --- called Seifert classes (see Definition \ref{Seifert}) --- such that their intersection product in $\widetilde{CH}^2((\As{4}{F}) \setminus Z)$ and its image by the boundary map $\partial : \widetilde{CH}^2((\As{4}{F}) \setminus Z) \to \CHMW{Z}{1}{0}{\det(\normalsheaf{Z}{\As{4}{F}})}$ --- called the quadratic linking class (see Definition \ref{defqlc}) --- only depend on the oriented link $Z$. Denoting by $\W(F)$ the Witt group of $F$, we call the image of the quadratic linking class by the isomorphism $\CHMW{Z}{1}{0}{\det(\normalsheaf{Z}{\As{4}{F}})} \to \W(F) \oplus \W(F)$ the quadratic linking degree (see Definition \ref{defqld}) and the image of the pushforward (by the inclusion) of the part of the quadratic linking class which is supported on $Z_1$ by the isomorphism $\CH{\As{4}{F}}{3}{1} \to \W(F)$ the ambient quadratic linking degree (see Definition \ref{defaqld}).
\end{theoremdefinition}

Let us illustrate this definition on the Hopf link $Z = \{x=y=0\} \sqcup \{z=t=0\} \subset \Spec(F[x,y,z,t]) \setminus \{0\}$ over a perfect field $F$ (see Example \ref{exHopf}). Its Seifert classes are the classes of $\langle x\rangle \otimes \overline{y}^*$ and $\langle z\rangle \otimes \overline{t}^*$ in $\widetilde{CH}^1((\As{4}{F}) \setminus Z)$, their intersection product is the class of $\langle x z\rangle \otimes (\overline{t}^* \wedge \overline{y}^*)$ in $\widetilde{CH}^2((\As{4}{F}) \setminus Z)$ and the quadratic linking class is the class of $-\langle z\rangle \eta \otimes (\overline{t}^* \wedge \overline{x}^* \wedge \overline{y}^*) \oplus \langle x\rangle \eta \otimes (\overline{y}^* \wedge \overline{z}^* \wedge \overline{t}^*)$ in $\CHMW{Z}{1}{0}{\det(\normalsheaf{Z}{\As{4}{F}})} \simeq \CHMW{Z_1}{1}{0}{\det(\normalsheaf{Z_1}{\As{4}{F}})} \oplus \CHMW{Z_2}{1}{0}{\det(\normalsheaf{Z_2}{\As{4}{F}})}$, which gives $(-1,1) \in \W(F) \oplus \W(F)$ as quadratic linking degree and $-1 \in \W(F)$ as ambient quadratic linking degree.

In Section \ref{secdef}, we give the definitions of the quadratic linking class, the quadratic linking degree and the ambient quadratic linking degree, then we determine how they depend on choices of orientations and of parametrisations of $\As{2}{F} \to \As{4}{F}$ (see Lemma \ref{stchanges}) and deduce invariants of the (ambient) quadratic linking degree (see Corollaries \ref{strankmod2} and \ref{stinvR} and Theorem \ref{stseriesofinvariants}). For instance, in the case $F = \R$, the absolute values of the components of the quadratic linking degree and of the ambient quadratic linking degree (which are in $\W(\R) \simeq \Z$) are invariant under changes of orientations and of parametrisations of $\As{2}{F} \to \As{4}{F}$. This is similar to the fact that the absolute value of the linking number does not depend on choices of orientations. In the general case, the ranks modulo $2$ of the components of the quadratic linking degree and of the ambient quadratic linking degree are invariants, and more importantly we have the following lemma-definition and theorem:

\begin{lemmadefinition}
		Let $d = \sum_{i=1}^n \langle a_i\rangle \in \W(F)$. There exists a unique sequence of abelian groups $Q_{d,k}$ and of elements $\Sigma_k(d) \in Q_{d,k}$, where $k$ ranges over the nonnegative even integers, such that $Q_{d,0} = \W(F)$, $\Sigma_0(d) = 1 \in Q_{d,0}$ and:
	\begin{itemize}
		\item for each positive even integer $k$, $Q_{d,k}$ is the quotient group $Q_{d,k-2}/(\Sigma_{k-2}(d))$;
		\item for each positive even integer $k$, $\Sigma_k(d) = \sum_{1 \leq i_1 < \dots < i_k \leq n} \langle \prod_{1 \leq j \leq k} a_{i_j}\rangle \in Q_{d,k}$.
	\end{itemize}
\end{lemmadefinition}

\begin{theorem}
	Let $\mathscr{L}$ be an oriented link with two components and $k > 0$ be even. We denote the quadratic linking degree of $\mathscr{L}$ by $\Qld_{\mathscr{L}} = (d_1,d_2) \in \W(F) \oplus \W(F)$ and the ambient quadratic linking degree by $\aqld{\mathscr{L}}$. Then $\Sigma_k(d_1)$, $\Sigma_k(d_2)$ and $\Sigma_k(\aqld{\mathscr{L}})$ are invariant under changes of orientation classes $\overline{o_1},\overline{o_2}$ and under changes of parametrisations of $\varphi_1, \varphi_2 : \As{2}{F} \to \As{4}{F}$.
\end{theorem}

Note that even though these invariants give the same value on $d_1,d_2$ and $\aqld{\mathscr{L}}$ in all of the examples in this paper, the author has found no reason for this to be true in general (over an arbitrary perfect field). The author conjectures that this is true in the real case, as this is a direct consequence of the conjecture that the real realization of the ambient quadratic linking degree is the linking number and that the real realization of the quadratic linking degree is the linking couple, at least up to a sign (this is a work in progress). In the classical case, the equality between each component of the linking couple and the linking number up to a sign results from the fact that the linking number is the image of each component of the linking couple by a surjective group morphism $\Z \to \Z$ (i.e. $\W(\R) \to \W(\R)$) and from the fact that there are only two such morphisms: the identity on $\Z$ and its opposite. On the other hand, for an arbitrary perfect field $F$, the ambient quadratic linking degree is the image of each component of the quadratic linking degree by a group morphism $\W(F) \to \W(F)$ which may or may not be surjective, and in addition there may be many surjective group morphisms $\W(F) \to \W(F)$.

In Section \ref{sechowto} we give a method to explicitly compute the quadratic linking class (see Theorem \ref{thmqlc}), the quadratic linking degree (see Theorem \ref{thmqld}) and the ambient quadratic linking degree (see Theorem \ref{thmaqld}) when the link $Z_1 \sqcup Z_2 \subset \As{4}{F}$ is such that for each $i \in \{1,2\}$ the closure $\overline{Z_i} \subset \A^4_F$ of $Z_i$ is given by two irreducible equations $\{f_i=0,g_i=0\}$ such that $\{g_1=0,g_2=0\}$ is of codimension $2$ in $(\As{4}{F}) \setminus (Z_1 \sqcup Z_2)$ and for each generic point $p$ of an irreducible component of $\{g_1=0,g_2=0\}$, $f_1$ and $f_2$ are units in the residue field $\kappa(p)$.

In Section \ref{seccompute} we compute the quadratic linking class, the quadratic linking degree and the ambient quadratic linking degree (as well as their invariants) on several examples. The Examples \ref{exbinary} (which we call binary links) showcase the usefulness of the invariant $\Sigma_2$ by showing that it can distinguish between an infinity of different links. The Examples \ref{extorus} are inspired by the torus links $T(2,2n)$ of linking number $n$ (the Hopf link if $n = 1$, the Solomon link if $n = 2$ and the $n$-gonal link (two intertwined $n$-gons) if $n \geq 3$).

In Appendix \ref{secresidue} we give an explicit definition (one which allows computations) of the residue morphisms of Milnor-Witt $K$-theory (see Theorem \ref{computeresidue}), which is used in Sections \ref{sechowto} and \ref{seccompute}.

In Appendix \ref{secprelim} we recall some useful notions about the Rost-Schmid complex and its groups.\\

\textbf{Acknowledgements.} The author thanks Fr\'ed\'eric D\'eglise and Adrien Dubouloz for their mentoring during her PhD, and more generally everyone (including the referees) who provided feedback on this work. The author was partially supported by Project ISITE-BFC ANR-15-IDEX-0008 \textquotedblleft Motivic Invariants of Algebraic Varieties\textquotedblright{} and ANR Project PRC \textquotedblleft HQDIAG\textquotedblright{} ANR-21-CE40-0015 \textquotedblleft Motivic Homotopy, Quadratic Invariants and Diagonal Classes\textquotedblright{}. The IMB received support from the EIPHI Graduate School (contract ANR-17-EURE-0002). The author was then supported by the Research Training Group 2553 \textquotedblleft Symmetries and Classifying Spaces: Analytic, Arithmetic and Derived\textquotedblright{}, funded by the German Research Foundation DFG.

\section{The quadratic linking degree}\label{secdef}

In this section, we define oriented links with two components, oriented fundamental classes (and cycles), Seifert classes (and divisors) relative to the link, the quadratic linking class and the (ambient) quadratic linking degree of the link. We then explicit how the quadratic linking class and the (ambient) quadratic linking degree depend on choices of orientations and parametrisations of $\As{2}{F} \to \As{4}{F}$ and deduce a series of invariants of the quadratic linking degree.

\subsection{Conventions and notations}\label{subconv}

Throughout this section, $F$ is a perfect field, we put $\A^2_F = \Spec(F[u,v])$, $\A^4_F = \Spec(F[x,y,z,t])$ and $X := \As{4}{F}$.

For $Z$ a smooth closed subscheme of a smooth scheme $Y$, we denote by $\normalsheaf{Z}{Y}$ the normal sheaf of $Z$ in $Y$, i.e. the dual of the $\Os_Z$-module $\I_Z/\I_Z^2$ with $\I_Z$ the ideal sheaf of $Z$ in $Y$. 

We denote the usual generators of the Milnor-Witt $K$-theory ring of $F$ by $[a] \in \K{F}{1}$ (with $a \in F^*$) and $\eta \in \K{F}{-1}$ (see \cite[Definition 3.1]{morel}). We put $\langle a\rangle := 1 + \eta [a] \in \K{F}{0}$.

For $Y$ a smooth finite-type $F$-scheme, $j \in \Z$ and $\Lb$ an invertible $\Os_Y$-module, we denote  the Rost-Schmid complex by $\Cm(Y,\KMWF{j}{\Lb})$ (see Definition \ref{defRScomp}) and the $i$-th Rost-Schmid group of this complex by $\CHMW{Y}{i}{j}{\Lb}$ (see Definition \ref{defRSgroups}) or simply $\CH{Y}{i}{j}$ if $\Lb = \Os_Y$.

We identify $\CH{\As{2}{F}}{1}{0}$ with $\W(F)$ via the (noncanonical) isomorphism $\zeta : \CH{\As{2}{F}}{1}{0} \to \W(F)$ which factorizes as follows:
\[\xymatrix{\CH{\As{2}{F}}{1}{0} \ar[r]^-{\partial} & \CHMW{\{0\}}{0}{-2}{\det(\normalsheaf{\{0\}}{\A^2_F})} \ar[r] & \K{F}{-2} \ar[r]^-{\eta^2 \mapsto 1} & \W(F)}\]
where $\partial$ is the boundary map (see Definition \ref{defboundarymap} and Theorem \ref{locseq}) and the map in the middle is induced by the isomorphism $\det(\normalsheaf{\{0\}}{\A^2_F}) \to \Os_{\{0\}} \otimes \Os_{\{0\}}$ which sends $\overline{u}^* \wedge \overline{v}^*$ to $1 \otimes 1$ (see Definition \ref{conviso}). We also identify $\CH{\As{4}{F}}{3}{2}$ with $\W(F)$ via the (noncanonical) isomorphism $\zeta' : \CH{\As{4}{F}}{3}{2} \to \W(F)$ which is defined in a similar manner to $\zeta$, with $\overline{x}^* \wedge \overline{y}^* \wedge \overline{z}^* \wedge \overline{t}^* \in \det(\normalsheaf{\{0\}}{\A^4_F})$ instead of $\overline{u}^* \wedge \overline{v}^* \in \det(\normalsheaf{\{0\}}{\A^2_F})$ (see Definition \ref{convisobis}).

\subsection{Definitions of the quadratic linking class and degree}\label{subdef}

In this subsection, we give a series of definitions which conclude with the definitions of the quadratic linking class and the (ambient) quadratic linking degree of an oriented link with two components.

In order to define oriented links with two components, we need the following definition (which was given by Morel in \cite{morel}).

\begin{definition}[Orientation of a locally free module]\label{orientation}
	An \emph{orientation} of a locally free module $\Vb$ of constant finite rank $r$ over an $F$-scheme $Y$ is an isomorphism $o : \det(\Vb) = \Lambda^r(\Vb) \to \Lb \otimes \Lb$ where $\Lb$ is an invertible $\Os_Y$-module.
	
	Two orientations $o : \det(\Vb) \to \Lb \otimes \Lb, o' : \det(\Vb) \to \Lb' \otimes \Lb'$ are said to be equivalent if there exists an isomorphism $\psi : \Lb \to \Lb'$ such that $(\psi \otimes \psi) \circ o = o'$. The equivalence class of $o$, denoted $\overline{o}$, is called the \emph{orientation class} of $o$.
\end{definition}

\begin{definition}[Oriented link with two components]\label{deflink}
	An \emph{oriented link} $\mathscr{L}$ with two components is the following data:
	\begin{itemize}
		\item a couple of closed immersions $\varphi_i : \As{2}{F} \to X$ with disjoint images $Z_i$;
		\item for $i \in \{1,2\}$, an orientation class $\overline{o_i}$ of the normal sheaf $\normalsheaf{Z_i}{X}$, represented by an isomorphism $o_i : \nu_{Z_i} := \det(\normalsheaf{Z_i}{X}) \to \Lb_i \otimes \Lb_i$.
		\end{itemize}
	We denote $Z := Z_1 \sqcup Z_2$, $\nu_Z := \det(\normalsheaf{Z}{X})$.
\end{definition}

\begin{remark}\label{rklink}
	The canonical morphisms $\psi_i : Z_i \to Z$ induce an isomorphism
	\[\psi_1^* \oplus \psi_2^* : \CHMW{Z}{i}{j}{\nu_Z} \to \CHMW{Z_1}{i}{j}{\nu_{Z_1}} \oplus \CHMW{Z_2}{i}{j}{\nu_{Z_2}}\]
	which allows us to identify $\CHMW{Z}{i}{j}{\nu_Z}$ with $\CHMW{Z_1}{i}{j}{\nu_{Z_1}} \oplus \CHMW{Z_2}{i}{j}{\nu_{Z_2}}$.
\end{remark}

\begin{definition}[Oriented fundamental class and cycles]\label{defclass}
	Let $\mathscr{L}$ be an oriented link with two components and $i \in \{1,2\}$. The \emph{oriented fundamental class} of the $i$-th component of $\mathscr{L}$, denoted by $[o_i]$, is the unique element of $\CHMW{Z_i}{0}{-1}{\nu_{Z_i}}$ which is sent to $\eta \in \CH{Z_i}{0}{-1}$ by the isomorphism $\CHMW{Z_i}{0}{-1}{\nu_{Z_i}} \to \CH{Z_i}{0}{-1}$ induced by $o_i$ (see Lemma \ref{propll}).
	
	Furthermore, an \emph{oriented fundamental cycle} of the $i$-th component of $\mathscr{L}$ is a representative in $\Cm^0(Z_i,\KMWF{-1}{\nu_{Z_i}})$ of the oriented fundamental class $[o_i]$.
\end{definition}

\begin{remark}
	Note that if $o_i$ and $o'_i$ represent the same orientation class then the isomorphism $\CHMW{Z_i}{0}{-1}{\nu_{Z_i}} \to \CH{Z_i}{0}{-1}$ induced by $o'_i$ is the same as the one induced by $o_i$, hence the oriented fundamental class $[o_i]$ only depends on the orientation class $\overline{o_i}$.
\end{remark}

Recall that the boundary map $\partial : \CH{X \setminus Z}{1}{1} \to \CHMW{Z}{0}{-1}{\nu_Z}$ is an isomorphism (see Definition \ref{defboundarymap} and Theorem \ref{locseq} and note that the groups $\CH{X}{1}{1}$ and $\CH{X}{2}{1}$ vanish (by Proposition \ref{RS})).

\begin{definition}[Seifert class and Seifert divisors]\label{Seifert}
	Let $\mathscr{L}$ be an oriented link with two components. The couple of \emph{Seifert classes} of $\mathscr{L}$ is the couple $(\seifert_{o_1},\seifert_{o_2})$, or $(\seifert_1,\seifert_2)$ for short, of elements of $\CH{X \setminus Z}{1}{1}$ such that $\partial(\seifert_1) = ([o_1],0)$ and $\partial(\seifert_2) = (0,[o_2])$.
	
	For $i \in \{1,2\}$, we call $\seifert_i$ the \emph{Seifert class} of $Z_i$ relative to the link $\mathscr{L}$. Furthermore, a \emph{Seifert divisor} of $Z_i$ relative to the link $\mathscr{L}$ is a representative in $\mathcal{C}^1(X \setminus Z,\KF{1})$ of $\seifert_i$.
\end{definition}

\begin{remark}
	For $i \in \{1,2\}$, the Seifert class of $Z_i$ relative to $\mathscr{L}$ depends on $Z$ and not only on $Z_i$ (and its orientation class $\overline{o_i}$). We could define a weaker notion of Seifert class of $Z_i$, which would only depend on $Z_i$ (and $\overline{o_i}$), but it is important for what follows to have this stronger notion of Seifert class.
\end{remark}

See Definition \ref{defintprod} for the intersection product, which is used in the following definition.

\begin{definition}[Quadratic linking class]\label{defqlc}
	Let $\mathscr{L}$ be an oriented link with two components. The \emph{quadratic linking class} of $\mathscr{L}$, denoted by $\qlc{\mathscr{L}}$, is the image of the intersection product of the Seifert class $\seifert_1$ with the Seifert class $\seifert_2$ by the boundary map $\partial : \CH{X \setminus Z}{2}{2} \to \CHMW{Z}{1}{0}{\nu_Z}$. We denote $\qlc{\mathscr{L}} = (\sigma_{1,\mathscr{L}},\sigma_{2,\mathscr{L}}) \in \CHMW{Z_1}{1}{0}{\nu_{Z_1}} \oplus \CHMW{Z_2}{1}{0}{\nu_{Z_2}}$ (see Remark \ref{rklink}).
\end{definition}

\begin{remark}\label{rkkernel}
	Note that the quadratic linking class $\qlc{\mathscr{L}}$ contains as much information as the intersection product $\seifert_1 \cdot \seifert_2$ since the boundary map $\partial : \CH{X \setminus Z}{2}{2} \to \CHMW{Z}{1}{0}{\nu_Z}$ is injective (see the localization long exact sequence (in Theorem \ref{locseq}) and note that the group $\CH{X}{2}{2}$ vanishes (by Proposition \ref{RS})). Also note that $\qlc{\mathscr{L}} \in \ker(i_*)$ since the image of $\partial$ is the kernel of $i_* : \CHMW{Z}{1}{0}{\nu_Z} \to \CH{X}{3}{2}$, where $i_*$ is the pushforward of the closed immersion $i : Z \to X$ (see the localization long exact sequence (in Theorem \ref{locseq})).
\end{remark}

\begin{notation}
	For $i \in \{1,2\}$, we denote by $\widetilde{o_i}$ the isomorphism $\CHMW{Z_i}{1}{0}{\nu_{Z_i}} \to \CH{Z_i}{1}{0}$ induced by $o_i$ (see Lemma \ref{propll}) and by $\varphi_i^*$ the isomorphism $\CH{Z_i}{1}{0} \to \CH{\As{2}{F}}{1}{0}$ induced by $\varphi_i$.
\end{notation}

Recall that we fixed an isomorphism $\zeta : \CH{\As{2}{F}}{1}{0} \to \W(F)$ (see Subsection \ref{subconv}).

\begin{definition}[Quadratic linking degree]\label{defqld}
	Let $\mathscr{L}$ be an oriented link with two components. The \emph{quadratic linking degree} of $\mathscr{L}$, denoted by $\qld{\mathscr{L}}$, is the image of the quadratic linking class of $\mathscr{L}$ by the isomorphism $(\zeta \oplus \zeta) \circ (\varphi_1^* \oplus \varphi_2^*) \circ (\widetilde{o_1} \oplus \widetilde{o_2}) : \CHMW{Z}{1}{0}{\nu_Z} \to \W(F) \oplus \W(F)$.
\end{definition}

\begin{notation}\label{notinclsubcomp}
	We denote by $(i_1)_*$ the inclusion of the subcomplex $\Cm^{\bullet - 2}(Z_1,\KMWF{0}{\nu_{Z_1}})$ in $\Cm^{\bullet}(X,\KF{2})$ (and the induced morphism in cohomology; see Remark \ref{rkinclsubcomp}) and by $(i_2)_*$ the inclusion of the subcomplex $\Cm^{\bullet - 2}(Z_2,\KMWF{0}{\nu_{Z_2}})$ in $\Cm^{\bullet}(X,\KF{2})$ (and the induced morphism in cohomology; see Remark \ref{rkinclsubcomp}).
\end{notation}

Recall that we fixed an isomorphism $\zeta' : \CH{\As{4}{F}}{3}{2} \to \W(F)$ (see Subsection \ref{subconv}).

\begin{definition}[Ambient quadratic linking degree]\label{defaqld}
	Let $\mathscr{L}$ be an oriented link with two components. The \emph{ambient quadratic linking degree} of $\mathscr{L}$, denoted by $\aqld{\mathscr{L}}$, is the element $\zeta'((i_1)_*(\sigma_{1,\mathscr{L}}))$ of $\W(F)$ (see Definition \ref{defqlc} and Notation \ref{notinclsubcomp}).
\end{definition}

\begin{remark}\label{rkotheraqld}
	Since $\qlc{\mathscr{L}} \in \ker(i_*)$ (see Remark \ref{rkkernel}) and $i_* = (i_1)_* \oplus (i_2)_*$ (as $Z = Z_1 \sqcup Z_2$), we have $(i_2)_*(\sigma_{2,\mathscr{L}}) = -(i_1)_*(\sigma_{1,\mathscr{L}})$ hence $\zeta'((i_2)_*(\sigma_{2,\mathscr{L}})) = -\zeta'((i_1)_*(\sigma_{1,\mathscr{L}})) = - \aqld{\mathscr{L}}$.
\end{remark}

\subsection{Invariants of the quadratic linking degree}\label{subinv}

By construction, the quadratic linking degree depends on choices of orientations and of parametrisations of $\As{2}{F} \to X$ and the ambient quadratic linking degree depends on choices of orientations. In this subsection we determine how these depend on such choices and construct invariants from them. 

Throughout this subsection, $\mathscr{L}$ is an oriented link with two components and we denote $\qld{\mathscr{L}} = (d_1,d_2) \in \W(F) \oplus \W(F)$.

We start by recalling how orientation classes can change.

\begin{lemma}\label{storientation}
	Let $i \in \{1,2\}$ and $o'_i : \nu_{Z_i} \to \Lb'_i \otimes \Lb'_i$ be an orientation of the normal sheaf of $Z_i$ in $X$.
	There exists $a \in F^*$ such that the orientation class of $o'_i$ is the orientation class of $o_i \circ (\times a)$.
\end{lemma}

\begin{proof}
	Recall that every invertible $\Os_{\A^2_F}$-module is isomorphic to $\Os_{\A^2_F}$ (since $\A^2_F$ is factorial) and that every invertible $\Os_{\As{2}{F}}$-module is the restriction of an invertible $\Os_{\A^2_F}$-module hence every invertible $\Os_{\As{2}{F}}$-module is isomorphic to $\Os_{\As{2}{F}}$. Since $Z_i \simeq \As{2}{F}$, there exist isomorphisms $\psi : \Lb_i \to \Os_{Z_i}$ and $\psi' : \Lb'_i \to \Os_{Z_i}$. From Definition \ref{orientation}, $\overline{(\psi \otimes \psi) \circ o_i} = \overline{o_i}$ and $\overline{(\psi' \otimes \psi') \circ o'_i} = \overline{o'_i}$. Denoting by $m : \Os_{Z_i} \otimes \Os_{Z_i} \to \Os_{Z_i}$ the multiplication, the morphism $m \circ ((\psi' \otimes \psi') \circ o'_i) \circ ((\psi \otimes \psi) \circ o_i)^{-1} \circ m^{-1}$ is an automorphism of $\Os_{Z_i}$ hence is the multiplication by an element of $\Gamma(Z_i,\Os_{Z_i}^*)$, i.e. by an element of $F^*$. The result follows directly.
\end{proof}

Recall that automorphisms of $\As{2}{F}$ are restrictions of automorphisms of $\A^2_F$ which preserve the origin (see \cite[Theorem 6.45 (Hartogs' theorem)]{goertzwedhorn}), hence they induce changes of coordinates of $\A^2_F = \Spec(F[u,v])$. We denote by $J_{\psi}$ the Jacobian determinant of an automorphism $\psi$ of $\As{2}{F}$; note that $J_{\psi}$ is in $F^*$ since $(F[u,v])^* = F^*$.

\begin{lemma}\label{stchanges} 
	Let $a = (a_1,a_2)$ be a couple of elements of $F^*$ and $\psi = (\psi_1,\psi_2)$ be a couple of automorphisms of $\As{2}{F}$.
	\begin{enumerate}
		\item Let $\mathscr{L}_a$ be the link obtained from $\mathscr{L}$ by changing the orientation class $\overline{o_1}$ into $\overline{o_1 \circ (\times a_1)}$ and the orientation class $\overline{o_2}$ into $\overline{o_2 \circ (\times a_2)}$. Then $\qlc{\mathscr{L}_a} = \langle a_1a_2\rangle \qlc{\mathscr{L}}$, $\qld{\mathscr{L}_a} = (\langle a_2\rangle d_1,\langle a_1\rangle d_2)$ and $\aqld{\mathscr{L}_a} = \langle a_1a_2\rangle \aqld{\mathscr{L}}$.
		\item Let $\mathscr{L}_\psi$ be the link obtained from $\mathscr{L}$ by changing $\varphi_1 : \As{2}{F} \to X$ into $\varphi_1 \circ \psi_1$ and $\varphi_2 : \As{2}{F} \to X$ into $\varphi_2 \circ \psi_2$. Then $\qlc{\mathscr{L}_\psi} = \qlc{\mathscr{L}}$, $\qld{\mathscr{L}_\psi} = (\langle J_{\psi_1}\rangle d_1, \langle J_{\psi_2}\rangle d_2)$ and $\aqld{\mathscr{L}_\psi} = \aqld{\mathscr{L}}$.
		\item Let $\mathscr{L}'$ be the link obtained from $\mathscr{L}$ by changing the order of the components. Then $\qlc{\mathscr{L}'} = - \qlc{\mathscr{L}}$, $\qld{\mathscr{L}'} = (-d_2,-d_1)$ and $\aqld{\mathscr{L}'} = \aqld{\mathscr{L}}$.
	\end{enumerate}
\end{lemma}

\begin{proof}
	
	(1) Note that for all $i \in \{1,2\}$, $[o_i \circ (\times a_i)] = \langle a_i^{-1}\rangle [o_i] = \langle a_i\rangle [o_i]$ hence, by Proposition \ref{propresidue} and Proposition \ref{intscalar}:
	\begin{align*}
	\seifert_{o_1 \circ (\times a_1)} = \langle a_1\rangle \seifert_{o_1} & \text{ and } \seifert_{o_2 \circ (\times a_2)} = \langle a_2\rangle \seifert_{o_2}\\
	\seifert_{o_1 \circ (\times a_1)} \cdot \seifert_{o_2 \circ (\times a_2)} & \ = \ \langle a_1a_2\rangle \seifert_{o_1} \cdot \seifert_{o_2} \\
	\partial(\seifert_{o_1 \circ (\times a_1)} \cdot \seifert_{o_2 \circ (\times a_2)}) & \ = \ \langle a_1a_2\rangle \partial(\seifert_{o_1} \cdot \seifert_{o_2}) \\
	\qlc{\mathscr{L}_a} & \ = \ \langle a_1a_2\rangle \qlc{\mathscr{L}}
	\end{align*}
	Note that $\widetilde{o_1 \circ (\times a_1)}(\langle a_1a_2\rangle \sigma_{1,\mathscr{L}}) = \langle a_1\rangle \widetilde{o_1}(\langle a_1a_2\rangle \sigma_{1,\mathscr{L}}) = \langle a_1^2a_2\rangle \widetilde{o_1}(\sigma_{1,\mathscr{L}}) = \langle a_2\rangle \widetilde{o_1}(\sigma_{1,\mathscr{L}})$ and similarly $\widetilde{o_2 \circ (\times a_2)}(\langle a_1a_2\rangle \sigma_{2,\mathscr{L}}) = \langle a_1\rangle \widetilde{o_2}(\sigma_{2,\mathscr{L}})$. It follows that $\qld{\mathscr{L}_a} = (\langle a_2\rangle d_1,\langle a_1\rangle d_2)$.
	
	The equality $\aqld{\mathscr{L}_a} = \langle a_1a_2\rangle \aqld{\mathscr{L}}$ follows from the equality $\qlc{\mathscr{L}_a} = \langle a_1a_2\rangle \qlc{\mathscr{L}}$. Indeed, $\aqld{\mathscr{L}_a}$ is by definition equal to $\zeta'((i_1)_*(\sigma_{1,\mathscr{L}_a}))$, hence is equal to $\zeta'((i_1)_*(\langle a_1a_2\rangle \sigma_{1,\mathscr{L}}))$. The result follows from the fact that $(i_1)_*$ (which is the inclusion of a subcomplex; see Notation \ref{notinclsubcomp}) and $\zeta'$ commute with product by $\langle a_1a_2\rangle$ (see Proposition \ref{propresidue} and Definition \ref{convisobis}).
	
	(2) From the definitions, $\qlc{\mathscr{L}_\psi} = \qlc{\mathscr{L}}$ and $(\widetilde{o_1} \oplus \widetilde{o_2})(\qlc{\mathscr{L}_\psi}) = (\widetilde{o_1} \oplus \widetilde{o_2})(\qlc{\mathscr{L}})$. Let $i \in \{1,2\}$. We denote by $\psi_i^* : \CH{\As{2}{F}}{1}{0} \to \CH{\As{2}{F}}{1}{0}$ the isomorphism induced by $\psi_i$. Note that $(\varphi_i \circ \psi_i)^*(\widetilde{o_i}(\sigma_{i,\mathscr{L}})) = \psi_i^*(\varphi_i^*(\widetilde{o_i}(\sigma_{i,\mathscr{L}})))$ and that the following diagram is commutative:
	\[\xymatrix{\CH{\As{2}{F}}{1}{0} \ar[r]^-\partial \ar[d]_-{\psi_i^*} & \CHMW{\{0\}}{0}{-2}{\det(\normalsheaf{\overline{\{0\}}}{\A^2_F})} \ar[d]^-{\psi_i^*} \\ \CH{\As{2}{F}}{1}{0} \ar[r]_-\partial & \CHMW{\{0\}}{0}{-2}{\det(\normalsheaf{\overline{\{0\}}}{\A^2_F})}}\]
	Hence $\partial((\varphi_i \circ \psi_i)^*(\widetilde{o_i}(\sigma_{i,\mathscr{L}}))) = \psi_i^*(\partial(\varphi_i^*(\widetilde{o_i}(\sigma_{i,\mathscr{L}}))))$.
	
	Finally note that for all $\alpha \in \K{F}{-2}$, $\psi_i^*(\alpha \otimes (\overline{u}^* \wedge \overline{v}^*)) = \langle J_{\psi_i}\rangle \alpha \otimes (\overline{u}^* \wedge \overline{v}^*)$. It follows from Definition \ref{conviso} that $\qld{\mathscr{L}_\psi} = (\langle J_{\psi_1}\rangle d_1, \langle J_{\psi_2}\rangle d_2)$. As for $\aqld{\mathscr{L}_\psi}$, it is clearly equal to $\aqld{\mathscr{L}}$ since $\varphi_1$ and $\varphi_2$ are not used in its definition (nor in the definition of the quadratic linking class).
	
	(3) By Proposition \ref{comintprod}, $\seifert_2 \cdot \seifert_1 = \langle -1\rangle (\seifert_1 \cdot \seifert_2)$ hence by Proposition \ref{propresidue}, $\partial(\seifert_2 \cdot \seifert_1) = \langle -1\rangle \partial(\seifert_1 \cdot \seifert_2)$, thus $\qlc{\mathscr{L}'} = \langle -1\rangle \qlc{\mathscr{L}} = - \qlc{\mathscr{L}}$ (since $\qlc{\mathscr{L}} \in \CHMW{Z}{1}{0}{\nu_Z}$ and for every field $k$ and $\alpha \in \K{k}{-1}$, $\langle -1\rangle \alpha = - \alpha$). It follows that $\qld{\mathscr{L}'} = (-d_2,-d_1)$ and that $\aqld{\mathscr{L}'} = \aqld{\mathscr{L}}$ (as $\zeta'((i_2)_*(\sigma_{2,\mathscr{L}})) = -\zeta'((i_1)_*(\sigma_{1,\mathscr{L}}))$ by Remark \ref{rkotheraqld}).
\end{proof}

We have proved in particular that the ambient quadratic linking degree and each component of the quadratic linking degree are each only multiplied by an $\langle a\rangle \in \W(F)$ with $a \in F^*$ when the orientation classes or the  parametrisations are changed, so that we get the following invariants.

\begin{corollary}\label{strankmod2}
	The rank modulo $2$ of $d_1$, the rank modulo $2$ of $d_2$ and the rank modulo $2$ of $\aqld{\mathscr{L}}$ are invariant under changes of orientation classes $\overline{o_1},\overline{o_2}$ and under changes of parametrisations of $\varphi_1, \varphi_2 : \As{2}{F} \to X$.
\end{corollary}

\begin{proof}
	For all $a \in F^*$, the rank modulo $2$ of an element of the Witt ring $\W(F)$ is invariant under the multiplication by $\langle a\rangle$. The result follows directly from Lemma \ref{stchanges}.
\end{proof}

Recall that $\W(\R) \simeq \Z$ (via the signature).

\begin{corollary}\label{stinvR}
	If $F = \R$ then the absolute value of $d_1$, the absolute value of $d_2$ and the absolute value of $\aqld{\mathscr{L}}$ are invariant under changes of orientation classes $\overline{o_1},\overline{o_2}$ and under changes of parametrisations of $\varphi_1, \varphi_2 : \As{2}{\R} \to X$.
\end{corollary}

\begin{proof}
	For all $a \in \R^*$, $\langle a\rangle = \langle 1\rangle = 1 \in \W(\R)$ or $\langle a\rangle = \langle -1\rangle = -1 \in \W(\R)$ since every real number is a square or the opposite of a square. The result follows directly from Lemma \ref{stchanges}.
\end{proof}

In the family of examples \ref{extorus}, we provide for each positive integer $n$ an example of an oriented link over $\R$ whose ambient quadratic linking degree has absolute value $n$ (and the same is true of each component of the quadratic linking degree); for examples of oriented links over $\R$ whose (ambient) quadratic linking degree is $0$, see Examples \ref{exbinary} (with $a < 0$ in these examples).

The following Lemma-Definition is an inductive definition. For each $d \in \W(F)$, with $k$ ranging over the nonnegative even integers, we define an abelian group $Q_{d,k}$ and an element $\disp \Sigma_k(d) \in Q_{d,k}$. In Theorem \ref{stseriesofinvariants} we will see that $\Sigma_k(d_1)$, $\Sigma_k(d_2)$ and $\Sigma_k(\aqld{\mathscr{L}})$ are invariants for even $k > 0$ (this is also true for $k = 0$ but uninteresting, as $\Sigma_0$ is a constant map on $\W(F)$ which we only defined for convenience); the assumption that $k$ is even is important for Theorem \ref{stseriesofinvariants}.

\begin{lemmadefinition}\label{stinvwd}
	Let $d \in \W(F)$. There exists a unique sequence of abelian groups $Q_{d,k}$ and of elements $\Sigma_k(d) \in Q_{d,k}$, where $k$ ranges over the nonnegative even integers, such that:
	\begin{itemize}
		\item $Q_{d,0} = \W(F)$ and $\Sigma_0(d) = 1 \in Q_{d,0}$;
		\item for each positive even integer $k$, $Q_{d,k}$ is the quotient group $Q_{d,k-2}/(\Sigma_{k-2}(d))$;
		\item for each positive even integer $k$, as soon as $n \in \N_0$ and $a_1,\dots,a_n \in F^*$ verify that $\disp \sum_{i=1}^n \langle a_i\rangle = d \in \W(F)$, we have $\disp \Sigma_k(d) = \sum_{1 \leq i_1 < \dots < i_k \leq n} \langle \prod_{1 \leq j \leq k} a_{i_j}\rangle \in Q_{d,k}$.
	\end{itemize}
\end{lemmadefinition}

\begin{remark}
	The uniqueness in the previous statement is clear, whereas the existence requires work (which is done below) since in $\W(F)$ the equality $\sum_{i=1}^n \langle a_i\rangle = \sum_{j=1}^m \langle b_j\rangle$ does not imply that the $a_i$ are equal to the $b_j$. Moreover, this equality does not imply that $\sum_{1 \leq i_1 < \dots < i_k \leq n} \langle \prod_{1 \leq j \leq k} a_{i_j}\rangle$ is equal to $\sum_{1 \leq p_1 < \dots < p_k \leq m} \langle \prod_{1 \leq q \leq k} b_{p_q}\rangle$ in $\W(F)$, which is why we need the abelian groups $Q_{d,k}$.
\end{remark}

\begin{proof}
	Recall the following presentation of the abelian group $\W(F)$: its generators are the $\langle a \rangle$ for $a \in F^*$ and its relations are the following:
	\begin{enumerate}
		\item $\langle ab^2\rangle = \langle a\rangle$ for all $a,b \in F^*$;
		\item $\langle a\rangle + \langle b\rangle = \langle a+b\rangle + \langle (a+b)ab\rangle$ for all $a,b \in F^*$ such that $a+b \neq 0$;
		\item $\langle 1\rangle + \langle -1\rangle = 0$.
	\end{enumerate} 
	We denote by $G$ the free abelian group of generators the $\langle a \rangle$ for $a \in F^*$, by $G_1$ the quotient of $G$ by the first relation above and by $G_2$ the quotient of $G_1$ by the second relation above, so that $\W(F)$ is the quotient of $G_2$ by the third relation above. 
	
	Let $k$ be a nonnegative even integer such that for all nonnegative even integers $l < k$, $Q_{d,l}$ is an abelian group and $\Sigma_l(d) \in Q_{d,l}$ which verify the conditions of the statement. Note that the quotient of the abelian group $Q_{d,k-2}$ by its subgroup $(\Sigma_{k-2}(d))$ is well-defined, so we can fix $Q_{d,k} = Q_{d,k-2}/(\Sigma_{k-2}(d))$. To show that $\Sigma_k(d)$ is well-defined (by the formula given in the statement), we proceed in four steps, in which we consider a representative of $d$ in $G$, in $G_1$, in $G_2$ and in $\W(F)$ (i.e. $d$ itself) respectively. Let $n \in \N_0$ and $a_1,\dots,a_n \in F^*$ be such that $\sum_{i=1}^n \langle a_i\rangle$ is a representative of $d$ in $G$. 
	
	\textbf{First step:} By definition of $G$, $\disp \sum_{1 \leq i_1 < \dots < i_k \leq n} \langle \prod_{1 \leq j \leq k} a_{i_j}\rangle$ is well-defined in $G$ hence in $Q_{d,k}$ (since $Q_{d,k}$ is obtained from $G$ by quotienting several times).
	
	\textbf{Second step:} For all $b \in F^*$, $\disp \sum_{2 \leq i_2 < \dots < i_k \leq n} \langle a_1 b^2 \prod_{2 \leq j \leq k} a_{i_j}\rangle \ \ + \sum_{2 \leq i_1 < \dots < i_k \leq n} \langle \prod_{1 \leq j \leq k} a_{i_j}\rangle = \sum_{1 \leq i_1 < \dots < i_k \leq n} \langle \prod_{1 \leq j \leq k} a_{i_j}\rangle$ in $G_1$ hence in $Q_{d,k}$ (since $Q_{d,k}$ is obtained from $G_1$ by quotienting several times) and similarly for other indices. Thus $\disp \sum_{1 \leq i_1 < \dots < i_k \leq n} \langle \prod_{1 \leq j \leq k} a_{i_j}\rangle \in Q_{d,k}$ only depends on the class of $\disp \sum_{i=1}^n \langle a_i\rangle$ in $G_1$.
	
	\textbf{Third step:} If $a_1 + a_2 \neq 0$ then in $G_2$:
	\begin{align*}
	& \sum_{3 \leq i_3 < \dots < i_k \leq n} \langle (a_1+a_2)^2 a_1 a_2 \prod_{3 \leq j \leq k} a_{i_j}\rangle + \sum_{3 \leq i_1 < \dots < i_k \leq n} \langle \prod_{1 \leq j \leq k} a_{i_j}\rangle \\
	& + \sum_{3 \leq i_2 < \dots < i_k \leq n} \langle (a_1+a_2) \prod_{2 \leq j \leq k} a_{i_j}\rangle + \sum_{3 \leq i_2 < \dots < i_k \leq n} \langle (a_1+a_2) a_1 a_2 \prod_{2 \leq j \leq k} a_{i_j}\rangle \\
	& = \sum_{3 \leq i_3 < \dots < i_k \leq n} \langle a_1 a_2 \prod_{3 \leq j \leq k} a_{i_j}\rangle + \sum_{3 \leq i_1 < \dots < i_k \leq n} \langle \prod_{1 \leq j \leq k} a_{i_j}\rangle \\
	& + (\langle a_1+a_2\rangle + \langle (a_1+a_2) a_1 a_2\rangle) \sum_{3 \leq i_2 < \dots < i_k \leq n} \langle \prod_{2 \leq j \leq k} a_{i_j}\rangle \\
	& = \sum_{3 \leq i_3 < \dots < i_k \leq n} \langle a_1 a_2 \prod_{3 \leq j \leq k} a_{i_j}\rangle + \sum_{3 \leq i_1 < \dots < i_k \leq n} \langle \prod_{1 \leq j \leq k} a_{i_j}\rangle \\
	& + (\langle a_1\rangle + \langle a_2\rangle) \sum_{3 \leq i_2 < \dots < i_k \leq n} \langle \prod_{2 \leq j \leq k} a_{i_j}\rangle \\
	& = \sum_{1 \leq i_1 < \dots < i_k \leq n} \langle \prod_{1 \leq j \leq k} a_{i_j}\rangle
	\end{align*}
	and similarly for other indices. This is true in $Q_{d,k}$ since $Q_{d,k}$ is obtained from $G_2$ by quotienting several times. Thus
	$\disp \sum_{1 \leq i_1 < \dots < i_k \leq n} \langle \prod_{1 \leq j \leq k} a_{i_j}\rangle \in Q_{d,k}$ only depends on the class of $\disp \sum_{i=1}^n \langle a_i\rangle$ in $G_2$.

	\textbf{Fourth step:} With the convention that $\disp \sum_{1 \leq i_3 < \dots < i_2 \leq n}  \langle \prod_{3 \leq j \leq 2} a_{i_j}\rangle = 1$, note that
	\[\sum_{1 \leq i_1 < \dots < i_k \leq n} \langle \prod_{1 \leq j \leq k} a_{i_j}\rangle + (\langle 1\rangle + \langle -1\rangle) \sum_{1 \leq i_2 < \dots < i_k \leq n} \langle \prod_{2 \leq j \leq k} a_{i_j}\rangle + \langle -1\rangle \sum_{1 \leq i_3 < \dots < i_k \leq n} \langle \prod_{3 \leq j \leq k} a_{i_j}\rangle\]
	is equal to $\disp \sum_{1 \leq i_1 < \dots < i_k \leq n} \langle \prod_{1 \leq j \leq k} a_{i_j}\rangle \ \ - \sum_{1 \leq i_3 < \dots < i_k \leq n} \langle \prod_{3 \leq j \leq k} a_{i_j}\rangle$ in $\W(F)$ hence in $Q_{d,k}$ (since $Q_{d,k}$ is obtained from $\W(F)$ by quotienting several times). Since $\disp \sum_{1 \leq i_3 < \dots < i_k \leq n} \langle \prod_{3 \leq j \leq k} a_{i_j}\rangle = \Sigma_{k-2}(d) = 0$ in $Q_{d,k}$ (by definition of $Q_{d,k}$), $\disp \sum_{1 \leq i_1 < \dots < i_k \leq n} \langle \prod_{1 \leq j \leq k} a_{i_j}\rangle \in Q_{d,k}$ only depends on the class of $\disp \sum_{i=1}^n \langle a_i\rangle$ in $\W(F)$, i.e. on $d$. Thus we can fix $\disp \Sigma_k(d) = \sum_{1 \leq i_1 < \dots < i_k \leq n} \langle \prod_{1 \leq j \leq k} a_{i_j}\rangle \in Q_{d,k}$.
\end{proof}

It follows from Lemma-Definition \ref{stinvwd} that we have for each positive even integer $k$ a map $\Sigma_k : \W(F) \to \bigcup_{d \in \W(F)} Q_{d,k}$ which verifies that for all $d \in \W(F)$, $\Sigma_k(d) \in Q_{d,k}$. How interesting the map $\Sigma_2 : \W(F) \to \W(F)/(1)$ is depends on the ground field $F$. For instance, if $F = \R$ then the map $\Sigma_2 : \W(\R) \simeq \Z \to \W(\R)/(1) = {0}$ is uninteresting since it is constant (but in the real case, Corollary \ref{stinvR} already provides the best possible invariant) whereas if $F = \Q$ then the map $\Sigma_2 : \W(\Q) \simeq \W(\R) \oplus \bigoplus_{p \in P} \W(\Z/p\Z) \to \W(\Q)/(1) \simeq \bigoplus_{p \in P} \W(\Z/p\Z)$ (with $P$ the set of prime numbers) is very interesting (see the discussion at the end of Examples \ref{exbinary}). This provides new invariants of the (ambient) quadratic linking degree: given two oriented links $\mathscr{L}_1$ and $\mathscr{L}_2$, one first compares $\Sigma_2(\aqld{\mathscr{L}_1}) \in \W(F)/(1)$ with $\Sigma_2(\aqld{\mathscr{L}_2}) \in \W(F)/(1)$; if this was not enough to distinguish $\mathscr{L}_1$ from $\mathscr{L}_2$, i.e. if $\Sigma := \Sigma_2(\aqld{\mathscr{L}_1}) = \Sigma_2(\aqld{\mathscr{L}_2})$, then one compares $\Sigma_4(\aqld{\mathscr{L}_1}) \in (\W(F)/(1))/(\Sigma)$ with $\Sigma_4(\aqld{\mathscr{L}_2}) \in (\W(F)/(1))/(\Sigma)$, and so on.

\begin{theorem}\label{stseriesofinvariants}
	Let $\mathscr{L}$ be an oriented link with two components and $k$ be a positive even integer. We denote the quadratic linking degree of $\mathscr{L}$ by $\Qld_{\mathscr{L}} = (d_1,d_2) \in \W(F) \oplus \W(F)$. Then $\Sigma_k(d_1)$, $\Sigma_k(d_2)$ and $\Sigma_k(\aqld{\mathscr{L}})$ are invariant under changes of orientation classes $\overline{o_1},\overline{o_2}$ and under changes of parametrisations of $\varphi_1, \varphi_2 : \As{2}{F} \to X$.
\end{theorem}

\begin{proof}
	Let $\disp \sum_{i = 1}^n \langle a_i\rangle \in \W(F)$. Note that for all $b \in F^*$:
	\[\Sigma_k(\langle b\rangle \sum_{i = 1}^n \langle a_i\rangle) = \sum_{1 \leq i_1 < \dots < i_k \leq n} \langle b^k \prod_{1 \leq j \leq k} a_{i_j}\rangle = \sum_{1 \leq i_1 < \dots < i_k \leq n} \langle \prod_{1 \leq j \leq k} a_{i_j}\rangle = \Sigma_k(\sum_{i = 1}^n \langle a_i\rangle)\]
	since $b^k$ is a square as $k$ is even. The result follows directly from Lemma \ref{stchanges}.
\end{proof}

\section{How to compute the quadratic linking degree}\label{sechowto}

In this section, we give a method to compute the quadratic linking class, the quadratic linking degree and the ambient quadratic linking degree under reasonable assumptions on the link (which are verified in the examples of Section \ref{seccompute}). See Subsection \ref{subconv} for notations and Subsection \ref{subdef} for definitions.

\subsection{Assumptions}\label{ass}

Let $\mathscr{L}$ be an oriented link with two components such that for all $i \in \{1,2\}$, the closure $\overline{Z_i} \subset \A^4_F$ of $Z_i$ is given by two equations
\[f_i(x,y,z,t) = 0, g_i(x,y,z,t) = 0\]
with $f_i$ and $g_i$ irreducible. We also assume that the subscheme of $X \setminus Z$ given by the equations $g_1=0$ and $g_2=0$ is of codimension $2$ in $X \setminus Z$ and that for each generic point $p$ of an irreducible component of this subscheme, $f_1$ and $f_2$ are units in the residue field $\kappa(p)$.

Let $i \in \{1,2\}$. Note that we can define an orientation of $\normalsheaf{Z_i}{X}$ from the (ordered) couple $(f_i,g_i)$. Indeed, $\normalsheaf{Z_i}{X}$ is the dual of the conormal sheaf $\cono{Z_i}{X} = \I_{Z_i}/\I_{Z_i}^2$, where $\I_{Z_i}$ is the ideal sheaf of $Z_i$ in $X$, and we have the following short exact sequence (see \cite[B.7.4]{fulton}):
\[\xymatrix{0 \ar[r] & (\cono{V(g_i)}{\A^4_F})_{|Z_i} \ar[r] & \cono{Z_i}{X} = (\cono{\overline{Z_i}}{\A^4_F})_{|Z_i} \ar[r] & (\cono{V(f_i)}{\A^4_F})_{|Z_i} \ar[r]  & 0}\]

We define the orientation $o_{(f_i,g_i)}$ as the isomorphism $\nu_{Z_i} \to \Os_{Z_i} \otimes \Os_{Z_i}$ which sends $\overline{f_i}^* \wedge \overline{g_i}^*$ to $1 \otimes 1$. By Lemma \ref{storientation}, there exists $a_i \in F^*$ such that $\overline{o_i} = \overline{o_{(f_i,g_i)} \circ (\times a_i)} = \overline{o_{(a_i^{-1}f_i,g_i)}}$. Without loss of generality (since we can replace $f_i$ with $a_i^{-1}f_i$), we assume that $o_i = o_{(f_i,g_i)}$.

\subsection{Notations}\label{not}

We denote by $\chi^{\odd} : \Z \to \{0,1\}$ the characteristic function of the set of odd numbers.

We denote $\epsilon := - \langle -1 \rangle$ and for all $n \in \N_0$, $n_\epsilon := \sum_{i=1}^n \langle (-1)^{i-1} \rangle$ and $(-n)_\epsilon := \epsilon \, n_\epsilon$.

In order to make explicit computations, we introduce the following notations. Note that the quadratic linking class, the quadratic linking degree and the ambient quadratic linking degree of $\mathscr{L}$ which are computed in Theorems \ref{thmqlc}, \ref{thmqld} and \ref{thmaqld} respectively do not depend on the choices of uniformizing parameters $\pi_p, \pi_{p,q}, \pi_{p,q,0}, \pi'_{p,q,0}$ made below, since these are not used in their Definitions (recall Definitions \ref{defqlc}, \ref{defqld} and \ref{defaqld}). 

We denote by $I$ the set of generic points of irreducible components of the subscheme of $X \setminus Z$ given by the equations $g_1=0$ and $g_2=0$.

For every $p \in I$, we denote by $\pi_p$ a uniformizing parameter of the discrete valuation ring $\Os_{X \setminus Z,p}/(g_1)$, by $u_p$ a unit in $\Os_{X \setminus Z,p}/(g_1)$ and by $m_p \in \Z$ such that $g_2 = u_p \pi_p^{m_p} \in \Os_{X \setminus Z,p}/(g_1)$.

For every $p \in I$ and $q \in \overline{\{p\}}^{(1)} \cap Z$, we denote by $\pi_{p,q}$ a uniformizing parameter of the discrete valuation ring $\Os_{\overline{\{p\}},q}$, by $u_{p,q}$ a unit in $\Os_{\overline{\{p\}},q}$ and by $m_{p,q} \in \Z$ such that $f_1 f_2 u_p = u_{p,q} \pi_{p,q}^{m_{p,q}} \in \Os_{\overline{\{p\}},q}$.

For every $i \in \{1,2\}$, $p \in I$ and $q \in \overline{\{p\}}^{(1)} \cap Z_i$, we denote by $\tau_{p,q} \in \nu_{q}$ such that $\overline{\pi_{p,q}}^* \otimes \overline{\pi_p}^* \otimes \overline{g_1}^* = \tau_{p,q} \otimes (\overline{f_i}^* \wedge \overline{g_i}^*)$, by $v_{p,q,0}$ the discrete valuation of $\Os_{\overline{\{\varphi_i^{-1}(q)\}},0}$ and by $\pi_{p,q,0}$ a uniformizing parameter for $v_{p,q,0}$. Note that such a $\tau_{p,q}$ exists since $\overline{\pi_{p,q}}^* \otimes \overline{\pi_p}^* \otimes \overline{g_1}^* \in  \Z[(\nu_{p,q} \otimes_{\kappa(q)} (\nu_{Z_i})_{|q}) \setminus \{0\}]$.

For every $i \in \{1,2\}$, $p \in I$ and $q \in \overline{\{p\}}^{(1)} \cap Z_i$, we let $(u_{p,q,0},m_{p,q,0}) \in \Os^*_{\overline{\{\varphi_i^{-1}(q)\}},0} \times \Z$ be the unique couple such that $\varphi_i^*(\overline{u_{p,q}}) = u_{p,q,0} \pi_{p,q,0}^{m_{p,q,0}}$ and we denote by $\lambda_{p,q,0} \in \K{F}{0}$ such that $\eta^2 \otimes (\overline{\pi_{p,q,0}}^* \otimes \varphi_i^*(\tau_{p,q})) = \lambda_{p,q,0} \, \eta^2 \otimes (\overline{u}^* \wedge \overline{v}^*)$. Note that such a $\lambda_{p,q,0}$ exists since $\overline{\pi_{p,q,0}}^* \otimes \varphi_i^*(\tau_{p,q}) \in \Z[(\det(\normalsheaf{\{0\}}{\A^2_F})_{|0}) \setminus \{0\}]$.

For every $p \in I$ and $q \in \overline{\{p\}}^{(1)} \cap Z_1$, we denote by $v'_{p,q,0}$ the discrete valuation of $\Os_{\overline{\{q\}},0}$ and by $\pi'_{p,q,0}$ a uniformizing parameter for $v'_{p,q,0}$; we let $(u'_{p,q,0},m'_{p,q,0}) \in \Os^*_{\overline{\{q\}},0} \times \Z$ be the unique couple such that $\overline{u_{p,q}} = u'_{p,q,0} (\pi'_{p,q,0})^{m'_{p,q,0}}$ and we denote by $\lambda'_{p,q,0} \in \K{F}{0}$ such that $\eta^2 \otimes (\overline{\pi'_{p,q,0}}^* \otimes \overline{\pi_{p,q}}^* \otimes \overline{\pi_p}^* \otimes \overline{g_1}^*) = \lambda'_{p,q,0} \, \eta^2 \otimes (\overline{x}^* \wedge \overline{y}^* \wedge \overline{z}^* \wedge \overline{t}^*)$. Note that such a $\lambda'_{p,q,0}$ exists since $\overline{\pi'_{p,q,0}}^* \otimes \overline{\pi_{p,q}}^* \otimes \overline{\pi_p}^* \otimes \overline{g_1}^* \in \Z[(\det(\normalsheaf{\{0\}}{\A^4_F})_{|0}) \setminus \{0\}]$.

\subsection{Computing the quadratic linking class and degree}\label{subcompqlc}

\begin{theorem}\label{thmqlc}
	Under the assumptions of Subsection \ref{ass} and with the notations in Subsection \ref{not}, the cycle
	\[\sum_{p \in I} \sum_{q \in \overline{\{p\}}^{(1)} \cap Z} \langle \overline{u_{p,q}}\rangle \eta \, \chi^{\odd}(m_p m_{p,q}) \otimes (\overline{\pi_{p,q}}^* \otimes \overline{\pi_p}^* \otimes \overline{g_1}^*)\]
	where $\langle \overline{u_{p,q}}\rangle \eta \, \chi^{\odd}(m_p m_{p,q}) \otimes (\overline{\pi_{p,q}}^* \otimes \overline{\pi_p}^* \otimes \overline{g_1}^*) \in \KS{\kappa(q)}{-1}{\nu_{q} \otimes (\nu_Z)_{|q}}$, represents the quadratic linking class of $\mathscr{L}$.
\end{theorem}

\begin{proof}
	From Definition \ref{defclass}, the oriented fundamental class $[o_i]$ is the class in $\CHMW{Z_i}{0}{-1}{\nu_{Z_i}}$ of $\eta \otimes (\overline{f_i}^* \wedge \overline{g_i}^*)$. It follows from Definition \ref{Seifert} and Theorem \ref{computeresidue} that the Seifert class $\seifert_i$ of $Z_i$ is the class in $\CH{X \setminus Z}{1}{1}$ of $\langle f_i\rangle \otimes \overline{g_i}^*$ (over the generic point $p_i$ of the hypersurface of $X \setminus Z$ of equation $g_i = 0$). In the expression above, $\langle f_i\rangle \in \K{\kappa(p_i)}{0}$ and $\overline{g_i}^* \in \Z[\det(\normalsheaf{\overline{\{p_i\}}}{X \setminus Z}) \setminus \{0\}]$; with a slight abuse of notation, we denoted by $f_i$ the image in the fraction field of $F[x,y,z,t]/(g_i)$ of $f_i \in F[x,y,z,t]$. We will make similar slight abuses of notation below.
	
	By Corollary \ref{intformula}, the intersection product of the Seifert class $\seifert_1$ of $Z_1$ with the Seifert class $\seifert_2$ of $Z_2$ is the class in $\CH{X \setminus Z}{2}{2}$ of the cycle:
	\[\sum_{p \in I} (m_p)_\epsilon \langle f_1f_2u_p\rangle \otimes (\overline{\pi_p}^* \otimes \overline{g_1}^*)\]
	
	The quadratic linking class is the image of this intersection product by the boundary map $\partial : \CH{X \setminus Z}{2}{2} \to \CHMW{Z}{1}{0}{\nu_Z}$ thus the cycle
	\[\sum_{p \in I} \sum_{q \in \overline{\{p\}}^{(1)} \cap Z}  (m_p)_\epsilon \partial^{\pi_{p,q}}_{v_q}(\langle f_1f_2u_p\rangle) \otimes (\overline{\pi_{p,q}}^* \otimes \overline{\pi_p}^* \otimes \overline{g_1}^*)\]
	represents the quadratic linking class (note that we used Proposition \ref{propresidue} to extract $(m_p)_\epsilon$ from the morphism $\partial^{\pi_{p,q}}_{v_q}$). By Theorem \ref{computeresidue} and Lemma \ref{chilemma}, the cycle
	\[\sum_{p \in I} \sum_{q \in \overline{\{p\}}^{(1)} \cap Z} \langle \overline{u_{p,q}}\rangle \eta \, \chi^{\odd}(m_p m_{p,q}) \otimes (\overline{\pi_{p,q}}^* \otimes \overline{\pi_p}^* \otimes \overline{g_1}^*)\]
	represents the quadratic linking class of $\mathscr{L}$.
\end{proof}

\begin{theorem}\label{thmqld}
	Under the assumptions of Subsection \ref{ass} and with the notations in Subsection \ref{not}, the quadratic linking degree of $\mathscr{L}$ is the following couple of elements of $\W(F)$:
	\[\left( \sum_{p \in I} \sum_{q \in \overline{\{p\}}^{(1)} \cap Z_1} \lambda_{p,q,0} \langle \overline{u_{p,q,0}}\rangle \, \chi^{\odd}(m_p m_{p,q} m_{p,q,0}),	\sum_{p \in I} \sum_{q \in \overline{\{p\}}^{(1)} \cap Z_2} \lambda_{p,q,0} \langle \overline{u_{p,q,0}}\rangle \, \chi^{\odd}(m_p m_{p,q} m_{p,q,0})\right) \]
\end{theorem}

\begin{proof}	
	Recall from Definition \ref{defqld} that the first step in computing the quadratic linking degree from the quadratic linking class consists in applying $\widetilde{o_1} \oplus \widetilde{o_2}$. It follows from Theorem \ref{thmqlc} and the assumption that for all $i \in \{1,2\}$, $o_i = o_{(f_i,g_i)}$ (see Subsection \ref{ass}) that the couple of cycles
	\[\left( \sum_{p \in I} \sum_{q \in \overline{\{p\}}^{(1)} \cap Z_1} \langle \overline{u_{p,q}}\rangle \eta \, \chi^{\odd}(m_p m_{p,q}) \otimes \tau_{p,q},\sum_{p \in I} \sum_{q \in \overline{\{p\}}^{(1)} \cap Z_2} \langle \overline{u_{p,q}}\rangle \eta \, \chi^{\odd}(m_p m_{p,q}) \otimes \tau_{p,q}\right) \]
	where $\langle \overline{u_{p,q}}\rangle \eta \, \chi^{\odd}(m_p m_{p,q}) \otimes \tau_{p,q} \in \KS{\kappa(q)}{-1}{\nu_{p,q}}$, represents $(\widetilde{o_1} \oplus \widetilde{o_2})(\Qlc_{\mathscr{L}})$.
		
	It follows that the couple of cycles
	\begin{align*}
	\left( \sum_{p \in I} \sum_{q \in \overline{\{p\}}^{(1)} \cap Z_1} \langle \varphi_1^*(\overline{u_{p,q}})\rangle \eta \, \chi^{\odd}(m_p m_{p,q}) \otimes \varphi_1^*(\tau_{p,q}),\right.\\
	\left. \sum_{p \in I} \sum_{q \in \overline{\{p\}}^{(1)} \cap Z_2} \langle \varphi_2^*(\overline{u_{p,q}})\rangle \eta \, \chi^{\odd}(m_p m_{p,q}) \otimes \varphi_2^*(\tau_{p,q})\right)
	\end{align*}
	where for all $i \in \{1,2\}$, $\langle \varphi_i^*(\overline{u_{p,q}})\rangle \eta \, \chi^{\odd}(m_p m_{p,q}) \otimes \varphi_i^*(\tau_{p,q}) \in \KS{\kappa(\varphi_i^{-1}(q))}{-1}{\nu_{\varphi_i^{-1}(q)}}$, represents $(\varphi_1^* \oplus \varphi_2^*)(\widetilde{o_1} \oplus \widetilde{o_2})(\Qlc_{\mathscr{L}})$. This is the second step in computing the quadratic linking degree (see Definition \ref{defqld}).
	
	Recall from Definition \ref{defqld} and Definition \ref{conviso} that the third step in computing the quadratic linking degree consists in applying the boundary map 
	\[\partial : \Cm^1(\As{2}{F},\KF{0}) \to \Cm^0(\{0\},\KMWF{-2}{\det(\normalsheaf{\{0\}}{\A^2_F})})\]
	to each element of the couple above, which gives:
	\begin{align*} 
	\left(\sum_{p \in I} \sum_{q \in \overline{\{p\}}^{(1)} \cap Z_1} \partial^{\pi_{p,q,0}}_{v_{p,q,0}}(\langle \varphi_1^*(\overline{u_{p,q}})\rangle) \eta \, \chi^{\odd}(m_p m_{p,q}) \otimes (\overline{\pi_{p,q,0}}^* \otimes \varphi_1^*(\tau_{p,q})),\right.\\
	\left.\sum_{p \in I} \sum_{q \in \overline{\{p\}}^{(1)} \cap Z_2} \partial^{\pi_{p,q,0}}_{v_{p,q,0}}(\langle \varphi_2^*(\overline{u_{p,q}})\rangle) \eta \, \chi^{\odd}(m_p m_{p,q}) \otimes (\overline{\pi_{p,q,0}}^* \otimes \varphi_2^*(\tau_{p,q})) \right) 
	\end{align*}
	where $\partial^{\pi_{p,q,0}}_{v_{p,q,0}}(\langle \varphi_i^*(\overline{u_{p,q}})\rangle) \eta \, \chi^{\odd}(m_p m_{p,q}) \otimes (\overline{\pi_{p,q,0}}^* \otimes \varphi_i^*(\tau_{p,q})) \in \KS{\kappa(0)}{-2}{\det(\normalsheaf{\{0\}}{\A^2_F})}$.
	
	By Theorem \ref{computeresidue}, for every $i \in \{1,2\}$ we have $\partial^{\pi_{p,q,0}}_{v_{p,q,0}}(\langle \varphi_i^*(\overline{u_{p,q}})\rangle) = \langle \overline{u_{p,q,0}}\rangle \eta \, \chi^{\odd}(m_{p,q,0})$ thus the third step gives:
	\begin{align*}
	\left( \sum_{p \in I} \sum_{q \in \overline{\{p\}}^{(1)} \cap Z_1} \langle \overline{u_{p,q,0}}\rangle \, \eta^2 \, \chi^{\odd}(m_p m_{p,q} m_{p,q,0}) \otimes (\overline{\pi_{p,q,0}}^* \otimes \varphi_1^*(\tau_{p,q})), \right.\\
	\left. \sum_{p \in I} \sum_{q \in \overline{\{p\}}^{(1)} \cap Z_2} \langle \overline{u_{p,q,0}}\rangle \, \eta^2 \, \chi^{\odd}(m_p m_{p,q} m_{p,q,0}) \otimes (\overline{\pi_{p,q,0}}^* \otimes \varphi_2^*(\tau_{p,q}))\right) 
	\end{align*}
	
	From Definition \ref{conviso} and the notations in Subsection \ref{not}, using the canonical isomorphism $\K{F}{-2} \simeq \W(F)$ (which sends $\eta^2$ to $1$), the final step gives:
	\[\left( \sum_{p \in I} \sum_{q \in \overline{\{p\}}^{(1)} \cap Z_1} \lambda_{p,q,0} \langle \overline{u_{p,q,0}}\rangle \, \chi^{\odd}(m_p m_{p,q} m_{p,q,0}),	\sum_{p \in I} \sum_{q \in \overline{\{p\}}^{(1)} \cap Z_2} \lambda_{p,q,0} \langle \overline{u_{p,q,0}}\rangle \, \chi^{\odd}(m_p m_{p,q} m_{p,q,0})\right) \]
\end{proof}

\begin{theorem}\label{thmaqld}
	Under the assumptions of Subsection \ref{ass} and with the notations in Subsection \ref{not}, the ambient quadratic linking degree of $\mathscr{L}$ is the following element of $\W(F)$:
	\[\sum_{p \in I} \sum_{q \in \overline{\{p\}}^{(1)} \cap Z_1} \lambda'_{p,q,0} \langle \overline{u'_{p,q,0}}\rangle \, \chi^{\odd}(m_p m_{p,q} m'_{p,q,0})\]
\end{theorem}

\begin{proof}
	By Definition \ref{defaqld}, the ambient quadratic linking degree of $\mathscr{L}$ is equal to $\zeta'((i_1)_*(\sigma_{1,\mathscr{L}}))$. Note that by Theorem \ref{thmqlc}, $\sigma_{1,\mathscr{L}} \in \CHMW{Z_1}{1}{0}{\nu_{Z_1}}$ is represented by the cycle
	\[\sum_{p \in I} \sum_{q \in \overline{\{p\}}^{(1)} \cap Z_1} \langle \overline{u_{p,q}}\rangle \eta \, \chi^{\odd}(m_p m_{p,q}) \otimes (\overline{\pi_{p,q}}^* \otimes \overline{\pi_p}^* \otimes \overline{g_1}^*)\]
	so that $(i_1)_*(\sigma_{1,\mathscr{L}}) \in \CH{X}{3}{2}$ is represented by the cycle
	\[\sum_{p \in I} \sum_{q \in \overline{\{p\}}^{(1)} \cap Z_1} \langle \overline{u_{p,q}}\rangle \eta \, \chi^{\odd}(m_p m_{p,q}) \otimes (\overline{\pi_{p,q}}^* \otimes \overline{\pi_p}^* \otimes \overline{g_1}^*)\]
	since $(i_1)_*$ is induced by the inclusion of the subcomplex $\Cm^{\bullet - 2}(Z_1,\KMWF{0}{\nu_{Z_1}})$ in $\Cm^{\bullet}(X,\KF{2})$ (see Notation \ref{notinclsubcomp}). We then apply (see Definition \ref{convisobis}) the boundary map $\partial :  \Cm^3(\As{4}{F},\KF{2}) \to \Cm^0(\{0\},\KMWF{-2}{\det(\normalsheaf{\{0\}}{\A^4_F})})$, which gives:
	\[\sum_{p \in I} \sum_{q \in \overline{\{p\}}^{(1)} \cap Z_1} \partial^{\pi'_{p,q,0}}_{v'_{p,q,0}}(\langle \overline{u_{p,q}}\rangle) \eta \, \chi^{\odd}(m_p m_{p,q}) \otimes (\overline{\pi'_{p,q,0}}^* \otimes \overline{\pi_{p,q}}^* \otimes \overline{\pi_p}^* \otimes \overline{g_1}^*)\]
	where $\partial^{\pi'_{p,q,0}}_{v'_{p,q,0}}(\langle \overline{u_{p,q}}\rangle) \eta \, \chi^{\odd}(m_p m_{p,q}) \otimes (\overline{\pi'_{p,q,0}}^* \otimes \overline{\pi_{p,q}}^* \otimes \overline{\pi_p}^* \otimes \overline{g_1}^*) \in \KS{\kappa(0)}{-2}{\det(\normalsheaf{\{0\}}{\A^4_F})}$.
	
	By Theorem \ref{computeresidue}, we have $\partial^{\pi'_{p,q,0}}_{v'_{p,q,0}}(\langle \overline{u_{p,q}}\rangle) = \langle \overline{u'_{p,q,0}}\rangle \eta \, \chi^{\odd}(m'_{p,q,0})$ thus $\partial((i_1)_*(\sigma_{1,\mathscr{L}}))$ is represented by the cycle:
	\[\sum_{p \in I} \sum_{q \in \overline{\{p\}}^{(1)} \cap Z_1} \langle \overline{u'_{p,q,0}}\rangle \eta^2 \, \chi^{\odd}(m_p m_{p,q} m'_{p,q,0}) \otimes (\overline{\pi'_{p,q,0}}^* \otimes \overline{\pi_{p,q}}^* \otimes \overline{\pi_p}^* \otimes \overline{g_1}^*)\]
	
	From Definition \ref{convisobis} and the notations in Subsection  \ref{not}, using the canonical isomorphism $\K{F}{-2} \simeq \W(F)$ (which sends $\eta^2$ to $1$), the ambient quadratic linking degree of $\mathscr{L}$ is equal to:
	\[\sum_{p \in I} \sum_{q \in \overline{\{p\}}^{(1)} \cap Z_1} \lambda'_{p,q,0} \langle \overline{u'_{p,q,0}}\rangle \, \chi^{\odd}(m_p m_{p,q} m'_{p,q,0})\]
\end{proof}

\section{Examples of computations of the quadratic linking degree}\label{seccompute}

In this section, we compute the quadratic linking class, the quadratic linking degree and the ambient quadratic linking degree (as well as their invariants) on examples. To do this we use the method given in Section \ref{sechowto}. See Section \ref{secdef} for definitions and notations.

\begin{example}{(Hopf)}\label{exHopf}
	We define the Hopf link over a perfect field $F$ as follows:
	\begin{itemize}
		\item $Z_1$ is the intersection of the closed subscheme of $\A^4_F = \Spec(F[x,y,z,t])$ of ideal $(x,y)$ and of $X := \As{4}{F}$;
		\item $\varphi_1 : \As{2}{F} \to X$ is the morphism associated to the morphism of $F$-algebras $F[x,y,z,t] \to F[u,v]$ which maps $x,y,z,t$ to $0,0,u,v$ respectively;
		\item $\overline{o_1}$ is the orientation class associated to the couple $(x,y)$ (i.e. the class of the isomorphism $o_1 : \nu_{Z_1} \to \Os_{Z_1} \otimes \Os_{Z_1}$ which maps $\overline{x}^* \wedge \overline{y}^*$ to $1 \otimes 1$);
		\item $Z_2$ is the intersection of the closed subscheme of $\A^4_F$ of ideal $(z,t)$ and of $X$;
		\item $\varphi_2 : \As{2}{F} \to X$ is the morphism associated to the morphism of $F$-algebras $F[x,y,z,t] \to F[u,v]$ which maps $x,y,z,t$ to $u,v,0,0$ respectively;
		\item $\overline{o_2}$ is the orientation class associated to the couple $(z,t)$ (i.e. the class of the isomorphism $o_2 : \nu_{Z_2} \to \Os_{Z_2} \otimes \Os_{Z_2}$ which maps $\overline{z}^* \wedge \overline{t}^*$ to $1 \otimes 1$).
	\end{itemize}

	In Table \ref{tabHopf} we give oriented fundamental cycles of $Z_1$ and $Z_2$, Seifert divisors of $Z_1$ (with orientation $o_1$) and $Z_2$ (with orientation $o_2$) relative to the link, their intersection product and its image by the boundary map $\partial : \CH{X \setminus Z}{2}{2} \to \CHMW{Z}{1}{0}{\nu_Z}$, which is the quadratic linking class (or rather we give a cycle which represents the quadratic linking class). Then we give cycles which represent $(\widetilde{o_1} \oplus \widetilde{o_2})(\Qlc_{\mathscr{L}})$, $(\varphi_1^* \oplus \varphi_2^*)((\widetilde{o_1} \oplus \widetilde{o_2})(\Qlc_{\mathscr{L}}))$, $(\partial \oplus \partial)((\varphi_1^* \oplus \varphi_2^*)((\widetilde{o_1} \oplus \widetilde{o_2})(\Qlc_{\mathscr{L}})))$ respectively and we give the quadratic linking degree (in $\W(F) \oplus \W(F)$). Finally, we give again the quadratic linking class, then a cycle which represents $(i_1)_*(\sigma_{1,\mathscr{L}})$, then a cycle which represents $\partial((i_1)_*(\sigma_{1,\mathscr{L}}))$ and then the ambient quadratic linking degree (in $\W(F)$). The points on which the cycles are supported are the obvious ones (for instance $\langle x\rangle \otimes \overline{y}^*$ is supported on the generic point of the hypersurface of $X \setminus Z$ of equation $y = 0$).
	
	Recall Theorems \ref{thmqlc}, \ref{thmqld} and \ref{thmaqld} (and their proofs) and note that the intersection of the hypersurfaces of $X \setminus Z$ of equations $y = 0$ and $t = 0$ is irreducible to get the results in Table \ref{tabHopf}.
	
	Note that the rank modulo $2$ of each component of the quadratic linking degree and of the ambient quadratic linking degree of the Hopf link is equal to $1$. Note that for every positive even integer $k$, the image by $\Sigma_k$ of each component of the quadratic linking degree and of the ambient quadratic linking degree of the Hopf link is equal to $0$. Note that if $F = \R$ then the absolute value of each component of the quadratic linking degree and of the ambient quadratic linking degree of the Hopf link is equal to $1$.
\end{example}

\begin{table} 
	\centering
	\renewcommand{\arraystretch}{1}
	\begin{tabular}{|l|ccc|}  
		\hline 
		
		Oriented fundamental cycles
		& $\eta \otimes (\overline{x}^* \wedge \overline{y}^*)$
		& |
		& $\eta \otimes (\overline{z}^* \wedge \overline{t}^*)$ \\
		\hline 
		
		Seifert divisors
		& $\langle x\rangle \otimes \overline{y}^*$
		& |
		& $\langle z\rangle \otimes \overline{t}^*$ \\
		\hline 
		
		Apply intersection product
		& \multicolumn{3}{c|}{\, \, \, \, \, $\langle x z\rangle \otimes (\overline{t}^* \wedge \overline{y}^*)$} \\
		\hline
		
		Quadratic linking class
		& $-\langle z\rangle \eta \otimes (\overline{t}^* \wedge \overline{x}^* \wedge \overline{y}^*)$ 
		& $\oplus$ 
		& $\langle x\rangle \eta \otimes (\overline{y}^* \wedge \overline{z}^* \wedge \overline{t}^*)$ \\ 
		\hline
		
		Apply $\widetilde{o_1} \oplus \widetilde{o_2}$
		& $-\langle z\rangle \eta \otimes \overline{t}^*$ 
		& $\oplus$ 
		& $\langle x\rangle \eta \otimes \overline{y}^*$ \\ 
		\hline
		
		Apply $\varphi_1^* \oplus \varphi_2^*$ 
		& $-\langle u\rangle \eta \otimes \overline{v}^*$ 
		& $\oplus$ 
		& $\langle u\rangle \eta \otimes \overline{v}^*$ \\ 
		\hline
		
		Apply $\partial \oplus \partial$
		& $-\eta^2 \otimes (\overline{u}^* \wedge \overline{v}^*)$ 
		& $\oplus$ 
		& $\eta^2 \otimes (\overline{u}^* \wedge \overline{v}^*)$ \\ 
		\hline
		
		Quadratic linking degree
		& $-1$ 
		& $\oplus$ 
		& $1$ \\ 
		\hline
		
		Quadratic linking class
		& $-\langle z\rangle \eta \otimes (\overline{t}^* \wedge \overline{x}^* \wedge \overline{y}^*)$ 
		& $\oplus$ 
		& $\langle x\rangle \eta \otimes (\overline{y}^* \wedge \overline{z}^* \wedge \overline{t}^*)$ \\ 
		\hline
		
		Apply $(i_1)_*$ to the part supp. on $Z_1$
		& $-\langle z\rangle \eta \otimes (\overline{t}^* \wedge \overline{x}^* \wedge \overline{y}^*)$
		&
		& \\
		\hline
		
		Apply $\partial$
		& $- \eta^2 \otimes (\overline{x}^* \wedge \overline{y}^* \wedge \overline{z}^* \wedge \overline{t}^*)$
		&
		& \\
		\hline
		
		Ambient quadratic linking degree
		& $-1$
		&
		& \\
		\hline
			
	\end{tabular}
	\caption{The Hopf link}
	\label{tabHopf}
\end{table}

Let us now present examples where the intersection of the underlying divisors is not irreducible (and where the invariants of Corollaries \ref{strankmod2} and \ref{stinvR} and of Theorem \ref{stseriesofinvariants} have different values).

\begin{examples}{(Binary links)}\label{exbinary} 
	Let $F$ be a perfect field of characteristic different from $2$ and $a \in F^*$. We define the binary link $B_a$ over $F$ as follows:
	\begin{itemize}
		\item $Z_1$ is the intersection of the closed subscheme of $\A^4_F$ of ideal $(f_1 := t - ((1+a)x-y)y, g_1 := z - x(x-y))$ and of $X := \As{4}{F}$;
		\item $\varphi_1 : \As{2}{F} \to X$ is the morphism associated to the morphism of $F$-algebras $F[x,y,z,t] \to F[u,v]$ which maps $x,y,z,t$ to $u,v,u(u-v),((1+a)u-v)v$ respectively;
		\item $\overline{o_1}$ is the orientation class associated to the couple $(f_1,g_1)$;
		\item $Z_2$ is the intersection of the closed subscheme of $\A^4_F$ of ideal $(f_2 := t + ((1+a)x-y)y,g_2 := z + x(x-y))$ and of $X$;
		\item $\varphi_2 : \As{2}{F} \to X$ is the morphism associated to the morphism of $F$-algebras $F[x,y,z,t] \to F[u,v]$ which maps $x,y,z,t$ to $u,v,-u(u-v),-((1+a)u-v)v$ respectively;
		\item $\overline{o_2}$ is the orientation class associated to the couple $(f_2,g_2)$.
	\end{itemize}

	In Table \ref{tabSolomon} we give oriented fundamental cycles of $Z_1$ and $Z_2$, Seifert divisors of $Z_1$ (with orientation $o_1$) and $Z_2$ (with orientation $o_2$) relative to the link, their intersection product and its image by the boundary map $\partial : \CH{X \setminus Z}{2}{2} \to \CHMW{Z}{1}{0}{\nu_Z}$, which is the quadratic linking class. Then we give cycles which represent $(\widetilde{o_1} \oplus \widetilde{o_2})(\Qlc_{\mathscr{L}})$, $(\varphi_1^* \oplus \varphi_2^*)((\widetilde{o_1} \oplus \widetilde{o_2})(\Qlc_{\mathscr{L}}))$, $(\partial \oplus \partial)((\varphi_1^* \oplus \varphi_2^*)((\widetilde{o_1} \oplus \widetilde{o_2})(\Qlc_{\mathscr{L}})))$ respectively and we give the quadratic linking degree. Finally, we give again the quadratic linking class, then a cycle which represents $(i_1)_*(\sigma_{1,\mathscr{L}})$, then a cycle which represents $\partial((i_1)_*(\sigma_{1,\mathscr{L}}))$ and then the ambient quadratic linking degree. Unless specified (between parentheses after a central dot), the points on which the cycles are supported are the obvious ones (for instance $\langle f_1\rangle \otimes \overline{g_1}^*$ is supported on the generic point of the hypersurface of $X \setminus Z$ of equation $g_1 = 0$).
	
	To see how one gets from the fifth line in Table \ref{tabSolomon} to the sixth line in this Table, note that $-\langle f_2\rangle \eta \otimes \overline{g_2}^* \cdot (x-y) \in \CH{Z_1}{1}{0}$ is equal to $-\langle 2((1+a)x-y)y\rangle \eta \otimes \overline{2x(x-y)}^* \cdot (x-y)$ since in $Z_1$: $t = ((1+a)x-y)y$ and $z = x(x-y)$. Further note that $-\langle 2((1+a)x-y)y\rangle \eta \otimes \overline{2x(x-y)}^* \cdot (x-y) = -\langle ((1+a)x-y)yx\rangle \eta \otimes \overline{x-y}^* \cdot (x-y)$ and that $-\langle ((1+a)x-y)yx\rangle \eta \otimes \overline{x-y}^* \cdot (x-y) = -\langle ax^3\rangle \eta \otimes \overline{x-y}^* \cdot (x-y) = -\langle ax\rangle \eta \otimes \overline{x-y}^* \cdot (x-y)$. Similarly, $-\langle f_2\rangle \eta \otimes \overline{g_2}^* \cdot (x) \in \CH{Z_1}{1}{0}$ is equal to $-\langle y\rangle \eta \otimes \overline{x}^* \cdot (x)$. A similar reasoning gets one from the tenth line to the eleventh line.
	
	Note that the rank modulo $2$ of each component of the quadratic linking degree and of the ambient quadratic linking degree of the binary link $B_a$ is $0$ (hence the invariant presented in Corollary \ref{strankmod2} distinguishes between the Hopf link and the binary links). Note that the image by $\Sigma_2$ of each component of the quadratic linking degree and of the ambient quadratic linking degree of the binary link $B_a$ is $\langle a\rangle \in \W(F)/(1)$. For instance, if $F = \Q$, $\Sigma_2$ distinguishes between all the $B_p$ with $p$ prime numbers since if $p \neq q$ are prime numbers then $\disp \langle p\rangle \in \W(\Q)/(1)$ corresponds to $1 \in \W(\Z/p\Z) \subset \bigoplus_{r \text{ prime}} \W(\Z/r\Z)$ and $\disp \langle q\rangle \in \W(\Q)/(1)$ corresponds to $1 \in \W(\Z/q\Z) \subset \bigoplus_{r \text{ prime}} \W(\Z/r\Z)$ (via the isomorphism $\W(\Q)/(1) \simeq \bigoplus_{r \text{ prime}} \W(\Z/r\Z)$ induced by the isomorphism $\W(\Q) \simeq \W(\R) \oplus \bigoplus_{r \text{ prime}} \W(\Z/r\Z)$ which maps $\langle p_1 \dots p_n\rangle \in \W(\Q)$ (with $p_1,\dots,p_n$ distinct primes) to $\langle p_1 \dots p_n\rangle \oplus \bigoplus_{i=1}^n \langle \prod_{j \neq i} p_j\rangle \in \W(\R) \oplus \bigoplus_{i=1}^n \W(\Z/p_i\Z)$). Note that if $F = \R$ then the absolute value of each component of the quadratic linking degree and of the ambient quadratic linking degree of the binary link $B_a$ is equal to $2$ if $a > 0$, to $0$ if $a < 0$ (hence the invariant presented in Corollary \ref{stinvR} distinguishes between the Hopf link and the binary links, as well as between the binary links with positive parameter and the binary links with negative parameter).
\end{examples}

\begin{table} 
	\centering
	\renewcommand{\arraystretch}{1}
	\begin{tabular}{|l|ccc|}  
		\hline 
		
		Or. fund. cyc.
		& $\eta \otimes (\overline{f_1}^* \wedge \overline{g_1}^*)$ 
		& |
		& $\eta \otimes (\overline{f_2}^* \wedge \overline{g_2}^*)$ \\
		\hline 
		
		Seifert divisors
		& $\langle f_1\rangle \otimes \overline{g_1}^*$
		& |
		& $\langle f_2\rangle \otimes \overline{g_2}^*$\\
		\hline 
		
		\multirow{2}{2cm}{Apply inter. prod.}
		& \multicolumn{3}{c|}{\phantom{+++}$\langle f_1 f_2\rangle \otimes (\overline{g_2}^* \wedge \overline{g_1}^*) \cdot (z,x-y)$} \\
		& \multicolumn{3}{c|}{$+ \langle f_1 f_2\rangle \otimes (\overline{g_2}^* \wedge \overline{g_1}^*) \cdot (z,x)$} \\
		\hline
		
		\multirow{2}{2cm}{Quad. link. class}
		& $-\langle f_2\rangle \eta \otimes (\overline{g_2}^* \wedge \overline{f_1}^* \wedge \overline{g_1}^*) \cdot (x-y)$
		& \multirow{2}{0.3cm}{$\oplus$} 
		& $\langle f_1\rangle \eta \otimes (\overline{g_1}^* \wedge \overline{f_2}^* \wedge \overline{g_2}^*) \cdot (x-y)$ \\ 
		
		& $-\langle f_2\rangle \eta \otimes (\overline{g_2}^* \wedge \overline{f_1}^* \wedge \overline{g_1}^*) \cdot (x) \phantom{++}$
		& & $+ \langle f_1\rangle \eta \otimes (\overline{g_1}^* \wedge \overline{f_2}^* \wedge \overline{g_2}^*) \cdot (x) \, \, \, \, \phantom{++}$ \\
		\hline
		
		\multirow{2}{1.5cm}{Apply $\widetilde{o_1} \oplus \widetilde{o_2}$}
		& $-\langle f_2\rangle \eta \otimes \overline{g_2}^* \cdot (x-y)$
		& \multirow{2}{0.3cm}{$\oplus$} 
		& $\phantom{+} \langle f_1\rangle \eta \otimes \overline{g_1}^* \cdot (x-y)$ \\
		
		& $-\langle f_2\rangle \eta \otimes \overline{g_2}^* \cdot (x) \phantom{++}$
		& & $+ \langle f_1\rangle \eta \otimes \overline{g_1}^* \cdot (x) \phantom{++}$ \\ 
		\hline
		
		\multirow{2}{1.5cm}{Apply $\varphi_1^* \oplus \varphi_2^*$}
		& $-\langle au\rangle \eta \otimes \overline{u-v}^* \cdot (u-v)$ 
		& \multirow{2}{0.3cm}{$\oplus$} 
		& $\phantom{+} \langle au\rangle \eta \otimes \overline{u-v}^* \cdot (u-v)$ \\ 
		
		& $-\langle v\rangle \eta \otimes \overline{u}^* \cdot (u) \phantom{+++}$ 
		& & $+ \langle v\rangle \eta \otimes \overline{u}^* \cdot (u) \phantom{+++}$ \\ 
		\hline
		
		Apply $\partial \oplus \partial$
		& $(1 + \langle a\rangle) \eta^2 \otimes (\overline{u}^* \wedge \overline{v}^*)$ 
		& $\oplus$ 
		& $-(1 + \langle a\rangle) \eta^2 \otimes (\overline{u}^* \wedge \overline{v}^*)$ \\ 
		\hline
		
		Quad. lk. deg.
		& $1 + \langle a\rangle$ 
		& $\oplus$ 
		& $-(1 + \langle a\rangle)$ \\ 
		\hline
		
		\multirow{2}{2cm}{Quad. link. class}
		& $-\langle f_2\rangle \eta \otimes (\overline{g_2}^* \wedge \overline{f_1}^* \wedge \overline{g_1}^*) \cdot (x-y)$
		& \multirow{2}{0.3cm}{$\oplus$} 
		& $\langle f_1\rangle \eta \otimes (\overline{g_1}^* \wedge \overline{f_2}^* \wedge \overline{g_2}^*) \cdot (x-y)$ \\ 
		
		& $-\langle f_2\rangle \eta \otimes (\overline{g_2}^* \wedge \overline{f_1}^* \wedge \overline{g_1}^*) \cdot (x) \phantom{++}$
		& & $+ \langle f_1\rangle \eta \otimes (\overline{g_1}^* \wedge \overline{f_2}^* \wedge \overline{g_2}^*) \cdot (x) \, \, \, \, \phantom{++}$ \\
		\hline
		
		\multirow{2}{2cm}{Apply $(i_1)_*$}
		& $-\langle f_2\rangle \eta \otimes (\overline{g_2}^* \wedge \overline{f_1}^* \wedge \overline{g_1}^*) \cdot (x-y)$ & & \\
		& $-\langle f_2\rangle \eta \otimes (\overline{g_2}^* \wedge \overline{f_1}^* \wedge \overline{g_1}^*) \cdot (x) \phantom{++}$ & & \\
		\hline
		
		\multirow{2}{2cm}{Apply $\partial$}
		& $-\langle a\rangle \eta^2 \otimes (\overline{x}^* \wedge \overline{y}^* \wedge \overline{z}^* \wedge \overline{t}^*)$ & & \\
		& $\phantom{aa} -\eta^2 \otimes (\overline{x}^* \wedge \overline{y}^* \wedge \overline{z}^* \wedge \overline{t}^*)$ & & \\
		\hline
		
		Ambient quad. lk. deg.
		& $-(1 + \langle a\rangle)$
		&
		& \\
		\hline
		
	\end{tabular}
	\caption{The binary link $B_a$}
	\label{tabSolomon}
\end{table}

The following family of examples is an analogue of the family of torus links $T(2,2n)$ (with $n \geq 1$ an integer) in knot theory. Note that $T(2,2)$ is the Hopf link and that its analogue below is slightly different from the Hopf link in the example above and has quadratic linking degree $(1,-1)$ and ambient quadratic linking degree $-1$. Note that $T(2,4)$ is the Solomon link.

\begin{examples}[Torus links]\label{extorus}
	Let $n \geq 1$. Let us define an analogue of the torus link $T(2,2n)$.
	
	Recall that (in knot theory) one of the components of $T(2,2n)$ is the intersection of $\{(a,b) \in \C^2, b = a^n\}$ with $\Sm^3_\varepsilon$, the $3$-sphere of radius $\varepsilon$, and that the other component of $T(2,2n)$ is the intersection of $\{(a,b) \in \C^2, b = - a^n\}$ with $\Sm^3_\varepsilon$ (for $\varepsilon > 0$ small enough). By writing $a = x + i y$ and $b = z + i t$ (with $x,y,z,t \in \R$), the equation $b = a^n$ becomes the system of equations
	\[\begin{cases}
	\disp t = \sum_{k=0}^{\lfloor \frac{n-1}{2} \rfloor} \binom{n}{2k+1} (-1)^k x^{n-2k-1} y^{2k+1} \\
		\disp z = \sum_{k=0}^{\lfloor \frac{n}{2} \rfloor} \binom{n}{2k} (-1)^k x^{n-2k} y^{2k}
	\end{cases}\]
	and the equation $b = - a^n$ becomes the system of equations
	\[\begin{cases}
	\disp t = -\sum_{k=0}^{\lfloor \frac{n-1}{2} \rfloor} \binom{n}{2k+1} (-1)^k x^{n-2k-1} y^{2k+1} \\
	\disp z = -\sum_{k=0}^{\lfloor \frac{n}{2} \rfloor} \binom{n}{2k} (-1)^k x^{n-2k} y^{2k}\\
	\end{cases}\] 
	From now on, we denote 
	\begin{align*}
	& \Sigma_t(x,y) := \sum_{k=0}^{\lfloor \frac{n-1}{2} \rfloor} \binom{n}{2k+1} (-1)^k x^{n-2k-1} y^{2k+1}, f_1 := t-\Sigma_t(x,y), f_2 := t+\Sigma_t(x,y), \\
	& \Sigma_z(x,y) := \sum_{k=0}^{\lfloor \frac{n}{2} \rfloor} \binom{n}{2k} (-1)^k x^{n-2k} y^{2k}, g_1 := z-\Sigma_z(x,y), g_2 := z+\Sigma_z(x,y)
	\end{align*}
	
	Consequently, we define our analogue over $\R$ of the torus link $T(2,2n)$ as follows:
	\begin{itemize}
		\item $Z_1$ is the intersection of the closed subscheme of $\A^4_\R$ of ideal $(f_1,g_1)$ and of $X := \As{4}{F}$;
		\item $\varphi_1 : \As{2}{\R} \to Z_1$ is the morphism associated to the morphism of $\R$-algebras $\R[x,y,z,t] \to \R[u,v]$ which maps $x,y,z,t$ to $u,v,\Sigma_z(u,v),\Sigma_t(u,v)$ respectively;
		\item $\overline{o_1}$ is the orientation class associated to the couple $(f_1,g_1)$;
		\item $Z_2$ is the intersection of the closed subscheme of $\A^4_\R$ of ideal $(f_2,g_2)$ and of $X$;
		\item $\varphi_2 : \As{2}{\R} \to Z_2$ is the morphism associated to the morphism of $\R$-algebras $\R[x,y,z,t] \to \R[u,v]$ which maps $x,y,z,t$ to $u,v,-\Sigma_z(u,v),-\Sigma_t(u,v)$ respectively;
		\item $\overline{o_2}$ is the orientation class associated to the couple $(f_2,g_2)$.
	\end{itemize}
	
	
	An oriented fundamental cycle of $Z_1$ (with orientation $o_1$) is $\eta \otimes (\overline{f_1}^* \wedge \overline{g_1}^*)$ (over the generic point of $Z_1$) and a Seifert divisor of $Z_1$ (with orientation $o_1$) is $\langle f_1\rangle \otimes \overline{g_1}^*$ (over the generic point of the hypersurface of $X \setminus Z$ of equation $g_1 = 0$).
	
	An oriented fundamental cycle of $Z_2$ (with orientation $o_2$) is $\eta \otimes (\overline{f_2}^* \wedge \overline{g_2}^*)$ (over the generic point of $Z_2$) and a Seifert divisor of $Z_2$ (with orientation $o_2$) is $\langle f_2\rangle \otimes \overline{g_2}^*$ (over the generic point of the hypersurface of $X \setminus Z$ of equation $g_2 = 0$).
	
	The intersection of the underlying divisors has $n$ irreducible components, whose generic points are denoted by $P_0,\dots,P_{n-1}$, where for all $j \in \{0,\dots,n-1\}$, the component of generic point $P_j$ is given in $X \setminus Z$ by the equations
	\[z=0,x=\tan\left(\frac{(n-1-2j)\pi}{2n}\right)y\]
	Indeed, if we denote $x+iy = \rho e^{i\theta}$ with $\rho \in \R_+^*, \theta \in \R$ then:
	\begin{align*}
	\Re((x+iy)^n) = 0 & \Leftrightarrow \cos(n\theta) = 0 \\
	& \Leftrightarrow \theta = \frac{(2j+1)\pi}{2n} \text{ for some } j \in \{0,\dots,2n-1\} \\
	& \Leftrightarrow x = \tan\left(\frac{(n-1-2j)\pi}{2n}\right) y \text{ for some } j \in \{0,\dots,n-1\}
	\end{align*}
	
	
	From now on, for every $j \in \{0,\dots,n-1\}$, we denote $\disp \theta_j := \frac{(n-1-2j)\pi}{2n}$. Thus, the homogeneous polynomial $\Sigma_z(x,y)$ of degree $n$ is equal to $\disp \prod_{j=0}^{n-1} (x - \tan(\theta_j) y)$. Note that the $\tan(\theta_j)$, with $j \in \{0,\dots,n-1\}$, are distinct, since they are the roots of the polynomial $(x+i)^n + (x-i)^n$ (which is coprime with its derivative).
	
	It follows (see Section \ref{sechowto}) that the intersection product of these Seifert divisors is equal to:
	\[\sum_{j=0}^{n-1} (m_j)_\epsilon \langle f_1 f_2 u_j\rangle \otimes (\overline{\pi_j}^* \wedge \overline{g_1}^*) \cdot (P_j)\]
	where $\pi_j$ (resp. $u_j$) is a uniformizing parameter (resp. a unit) in $\Os_{X \setminus Z, P_j}/(g_1)$ and $m_j \in \Z$ such that $g_2 = u_j \pi_j^{m_j}$. Note that one can choose $\pi_j = g_2$ (hence $m_j = 1$ and $u_j = 1$) since $\disp \Os_{X \setminus Z, P_j}/(g_1) \simeq (\R[x,y,z,t]/(z-\prod_{i=0}^{n-1} (x - \tan(\theta_i) y)))_{(z,x - \tan(\theta_j) y)} \simeq \R[x,y,t]_{(x-\tan(\theta_j)y)}$ and in this ring $\disp g_2 = 2 \prod_{i=0}^{n-1}(x-\tan(\theta_i)y)$, thus the intersection product of these Seifert divisors is equal to:
	\[\sum_{j=0}^{n-1}\langle f_1 f_2\rangle \otimes (\overline{g_2}^* \wedge \overline{g_1}^*) \cdot (P_j)\]
	
	It follows (see Section \ref{sechowto}) that its image by the boundary map, which is the quadratic linking class, is the following:
	\begin{align*}
	\sum_{j=0}^{n-1} -\langle f_2\rangle \eta \otimes (\overline{g_2}^* \wedge \overline{f_1}^* \wedge \overline{g_1}^*) \cdot (x = \tan(\theta_j)y \text{ in } Z_1) \\ + \sum_{j=0}^{n-1} \langle f_1\rangle \eta \otimes (\overline{g_1}^* \wedge \overline{f_2}^* \wedge \overline{g_2}^*) \cdot (x = \tan(\theta_j)y \text{ in } Z_2)
	\end{align*}
	
	Its image by $\widetilde{o_1} \oplus \widetilde{o_2}$ is:
	\[\sum_{j=0}^{n-1} -\langle f_2\rangle \eta \otimes \overline{g_2}^* \oplus \sum_{j=0}^{n-1} \langle f_1\rangle \eta \otimes \overline{g_1}^*\]
	
	Its image by $\varphi_1^* \oplus \varphi_2^*$ is:
	\[\sum_{j=0}^{n-1} -\langle 2\Sigma_t(u,v)\rangle \eta \otimes \overline{2\Sigma_z(u,v)}^* \oplus \sum_{j=0}^{n-1} \langle -2\Sigma_t(u,v)\rangle \eta \otimes \overline{-2\Sigma_z(u,v)}^*\] 
	
	Note that the first component of the couple above is equal to:
	\[\sum_{j=0}^{n-1} \langle -\sum_{k=0}^{\lfloor \frac{n-1}{2} \rfloor} \binom{n}{2k+1} (-1)^k (\tan(\theta_j))^{n-2k-1} \prod_{i \neq j, i = 0}^{n-1} (\tan(\theta_j) - \tan(\theta_i)) v\rangle \eta \otimes \overline{u - \tan(\theta_j) v}^*\]
	
	Its image by the boundary map $\partial$ is the following:
	\begin{align*}
	\sum_{j=0}^{n-1} \langle -\sum_{k=0}^{\lfloor \frac{n-1}{2} \rfloor} \binom{n}{2k+1} (-1)^k (\tan(\theta_j))^{n-2k-1} \prod_{i \neq j, i = 0}^{n-1} (\tan(\theta_j) - \tan(\theta_i))\rangle \eta^2 \otimes (\overline{v}^* \wedge \overline{u - \tan(\theta_j) v}^*) \\
	= \sum_{j=0}^{n-1} \langle \sum_{k=0}^{\lfloor \frac{n-1}{2} \rfloor} \binom{n}{2k+1} (-1)^k (\tan(\theta_j))^{n-2k-1} \prod_{i \neq j, i = 0}^{n-1} (\tan(\theta_j) - \tan(\theta_i))\rangle \eta^2 \otimes (\overline{u}^* \wedge \overline{v}^*) 
	\end{align*}

	Note that $\disp \sum_{k=0}^{\lfloor \frac{n-1}{2} \rfloor} \binom{n}{2k+1} (-1)^k (\tan(\theta_j))^{n-2k-1} = \Im((\tan(\theta_j) + i)^n) = \rho_j \sin(\frac{(2j+1)\pi}{2})$ with $\rho_j$ a positive real number, hence:
	\[\langle \sum_{k=0}^{\lfloor \frac{n-1}{2} \rfloor} \binom{n}{2k+1} (-1)^k (\tan(\theta_j))^{n-2k-1}\rangle = \begin{cases}
	\langle 1\rangle & \text{ if } j \text{ is even} \\ \langle -1\rangle & \text{ if } j \text{ is odd}
	\end{cases}\]
	
	Note that for all $l \in \{0,\dots,n-1\}$, $- \frac{\pi}{2} < \theta_l < \frac{\pi}{2}$ hence for all $i < j$, $\tan(\theta_j) - \tan(\theta_i) < 0$ and for all $i > j$, $\tan(\theta_j) - \tan(\theta_i) > 0$, hence:
	\[\langle \prod_{i \neq j, i = 0}^{n-1} (\tan(\theta_j) - \tan(\theta_i))\rangle = \begin{cases}
	\langle 1\rangle & \text{ if } j \text{ is even} \\ \langle -1\rangle & \text{ if } j \text{ is odd}
	\end{cases}\]
	
	Therefore $\partial(\varphi_1^*(\widetilde{o_1}(\sigma_{1,\mathscr{L}}))) = n \, \eta^2 \otimes (\overline{u}^* \wedge \overline{v}^*)$, hence the first component of the quadratic linking degree is equal to $n \in \W(\R)$.
	
	With similar computations to the ones above, we find that the second component of the quadratic linking degree is equal to $-n \in \W(\R)$, hence the quadratic linking degree is equal to $(n,-n) \in \W(\R) \oplus \W(\R) \simeq \Z \oplus \Z$ and we find that the ambient quadratic linking degree is equal to $-n \in \W(\R) \simeq \Z$.
	
	Note that the rank modulo $2$ of each component of the quadratic linking degree and of the ambient quadratic linking degree of the analogue of $T(2,2n)$ is $1$ if $n$ is odd, $0$ if $n$ is even. Note that the absolute value of each component of the quadratic linking degree and of the ambient quadratic linking degree of the analogue of $T(2,2n)$ is equal to $n$, hence the invariant presented in Corollary \ref{stinvR} distinguishes between all these links $T(2,2n)$, similarly to the absolute value of the linking number which distinguishes between all the links $T(2,2n)$ in knot theory (recall that the linking number of $T(2,2n)$ is equal to $n$ in classical knot theory).
\end{examples}

\appendix

\section{An explicit definition of the residue morphisms of Milnor-Witt $K$-theory}\label{secresidue}

In this appendix, we give an explicit definition (i.e. one which allows computations) of the noncanonical residue morphism and prove that it is indeed the noncanonical residue morphism (as defined by Morel in \cite{morel} and recalled in Definition \ref{defmorel}). Note that explicit definitions of the canonical residue morphism (see Definition \ref{canonicalres}) and of the twisted canonical residue morphism (see Definition \ref{twistedcanonicalres}) follow directly. We use the case $n \leq 0$ in Theorem \ref{computeresidue} to compute the quadratic linking class and degree in Sections \ref{sechowto} and \ref{seccompute}; the case $n \geq 1$ is included for its usefulness in other computations.

See \cite[Section 3.1]{morel} for recollections about Milnor-Witt $K$-theory. Throughout this Appendix, $F$ is a perfect field, $v : F^* \to \Z$ is a discrete valuation (of residue field $\kappa(v)$ and ring $\Os_v$) and $\pi$ is a uniformizing parameter for $v$. For all $u \in \Os_v^*$, we denote by $\overline{u}$ its class in $\kappa(v)$ (which is in $\kappa(v)^*$ since $u \in \Os_v^*$). We denote the usual generators of the Milnor-Witt $K$-theory ring of $F$ by $[a] \in \K{F}{1}$ (with $a \in F^*$) and $\eta \in \K{F}{-1}$ (see \cite[Definition 3.1]{morel}). We denote $\langle a\rangle := 1 + \eta [a] \in \K{F}{0}$, $\epsilon := - \langle -1 \rangle$ and for all $n \in \N_0$, $n_\epsilon := \sum_{i=1}^n \langle (-1)^{i-1} \rangle$ and $(-n)_\epsilon := \epsilon \, n_\epsilon$. We denote by $\chi^{\odd} : \Z \to \{0,1\}$ the characteristic function of the set of odd numbers.

We now recall Morel's definition of the noncanonical residue morphism.

\begin{definition}[The noncanonical residue morphism]\label{defmorel}
	The residue morphism $\partial^{\pi}_v : \K{F}{*} \to \K{\kappa(v)}{*-1}$ is the only morphism of graded groups which commutes to product by $\eta$ and satisfies, for all $n \in \N_0, u_1,\dots,u_n \in \Os_v^*$:
	\[\partial^{\pi}_v([\pi,u_1,\dots,u_n]) = [\overline{u_1},\dots,\overline{u_n}] \text{ and } \partial^{\pi}_v([u_1,\dots,u_n]) = 0.\]
	(For $n=0$, this means $\partial^{\pi}_v([\pi]) = 1$ and $\partial^{\pi}_v(1) = 0$.)
\end{definition}

In \cite[Theorem 3.15]{morel}, Morel proves that such a morphism exists and that it is unique.

Before we define the canonical residue morphism, we recall the following facts and definition:

\begin{proposition}[Proposition 3.17 in \cite{morel}]\label{propresidue}
	\begin{equation*}
	\forall u \in \Os_v^*, \forall \alpha \in \K{F}{*}, \ \partial^{\pi}_v(\langle u \rangle \alpha) = \langle \overline{u} \rangle \partial^{\pi}_v(\alpha)
	\end{equation*}
\end{proposition}

\begin{corollary}\label{corresidue}
	If $\pi' = u' \pi$ with $u' \in \Os_v^*$ then $\partial^{\pi}_v = \langle \overline{u'}\rangle \partial^{\pi'}_v$. 
\end{corollary}

\begin{definition}[Twisted Milnor-Witt $K$-theory]\label{deftwMW}
	Let $m \in \Z$ and $L$ be an $F$-vector space of dimension $1$. The $L$-twisted $m$-th Milnor-Witt $K$-theory abelian group of $F$, denoted $\KS{F}{m}{L}$, is the tensor product of the $\Z[F^*]$-modules $\K{F}{m}$ and $\Z[L \setminus \{0\}]$ (the scalar product of $\K{F}{m}$ being $(\sum_{f \in F^*} n_f \lambda_f) \cdot \alpha = \sum_{f \in F^*} n_f \langle f\rangle \alpha$).
\end{definition}

\begin{remark}
Note that if we fix an isomorphism between $L$ and $F$ then we get an isomorphism of $\Z[F^*]$-modules between $\KS{F}{m}{L}$ and $\K{F}{m}$; nevertheless, $\KS{F}{m}{L}$ is a useful construction because there is no canonical isomorphism between $L$ and $F$ (hence no canonical isomorphism between $\KS{F}{m}{L}$ and $\K{F}{m}$, unless $L = F$) and the introduction of $\KS{F}{m}{L}$ is what allows us to have canonical residue morphisms.
\end{remark}

\begin{definition}[The canonical residue morphism]\label{canonicalres}
	The canonical residue morphism $\partial_v : \K{F}{*} \to \KS{\kappa(v)}{*-1}{(\mathfrak{m}_v/\mathfrak{m}_v^2)^{\vee}}$ (where $\vee$ denotes the dual) is given by $\partial_v = \partial^{\pi}_v \otimes \overline{\pi}^*$ (with $\overline{\pi}$ the class of $\pi$ in $\mathfrak{m}_v/\mathfrak{m}_v^2$ (which is nonzero since $\pi$ is a uniformizing parameter for $v$) and $\overline{\pi}^*$ its dual basis).
\end{definition}

Note that $\partial_v$ does not depend on the choice of $\pi$, since if $\pi'$ is another uniformizing parameter for $v$ then there exists $u' \in \Os_v$ such that $\pi' = u' \pi$ hence, by Corollary \ref{corresidue} , $\partial^{\pi}_v \otimes \overline{\pi}^* = \langle \overline{u'}\rangle \partial^{\pi'}_v \otimes \overline{\pi}^* = \partial^{\pi'}_v \otimes \overline{u'\pi}^* = \partial^{\pi'}_v \otimes \overline{\pi'}^*$.

\begin{definition}[The twisted canonical residue morphism]\label{twistedcanonicalres}
	Let $L$ be a rank one $\Os_v$-module. The twisted canonical residue morphism $\partial_{v,L} : \KS{F}{*}{L \otimes_{\Os_v} F} \to \KS{\kappa(v)}{*-1}{(\mathfrak{m}_v/\mathfrak{m}_v^2)^{\vee} \otimes_{\kappa(v)} (L \otimes_{\Os_v} \kappa(v))}$
	is the morphism of graded groups which satisfies for all $\alpha \in \K{F}{*}$ and $l \in L$:
	\[\partial_{v,L}(\alpha \otimes (l \otimes 1)) = \partial^{\pi}_v(\alpha) \otimes (\overline{\pi}^* \otimes (l \otimes 1))\]
\end{definition}

Before we prove Theorem \ref{computeresidue}, we recall the following facts.

\begin{lemma}\label{epslemma}
	For all $m,n \in \Z$, $(mn)_{\epsilon} = m_{\epsilon}n_{\epsilon}$.
\end{lemma}

\begin{lemma}\label{chilemma}
	For all $m \in \Z$, $\eta m_{\epsilon} = \eta \chi^{\odd}(m)$.
\end{lemma}	

Recall that by \cite[Lemma 3.6]{morel}, for all $n \leq 0$, $\K{F}{n}$ is generated by elements of the form $\langle \pi^{m}u\rangle \eta^{-n}$ with $m \in \Z$ and $u \in \Os_v^*$, hence, since $\partial^{\pi}_v : \K{F}{n} \to \K{\kappa(v)}{n-1}$ is a group morphism (see Definition \ref{defmorel}), we only need to give $\partial^{\pi}_v(\langle \pi^{m}u\rangle \eta^{-n})$.

Recall that by \cite[Lemma 3.6]{morel}, for all $n \geq 1$, $\K{F}{n}$ is generated by elements of the form $[\pi^{m_1}u_1,\dots,\pi^{m_n}u_n]$ with $m_1,\dots,m_n \in \Z$ and $u_1,\dots,u_n \in \Os_v^*$, hence, since $\partial^{\pi}_v : \K{F}{n} \to \K{\kappa(v)}{n-1}$ is a group morphism (see Definition \ref{defmorel}), we only need to give $\partial^{\pi}_v([\pi^{m_1}u_1,\dots,\pi^{m_n}u_n])$.

\begin{theorem}\label{computeresidue}
	For all $n \leq 0$, $m \in \Z$ and $u \in \Os_v^*$:
	\[\partial^{\pi}_v(\langle \pi^{m}u\rangle \eta^{-n}) = \langle \overline{u}\rangle \eta^{-n+1} \chi^{\odd}(m)\]
	
	For all $n \geq 1$, $m_1,\dots,m_n \in \Z$ and $u_1,\dots,u_n \in \Os_v^*$:
	\begin{align*}
	& \partial^{\pi}_v([\pi^{m_1}u_1,\dots,\pi^{m_n}u_n]) = \\ & \sum_{l=0}^{n-1} \sum_{\substack{J \subset \{1,\dots,n\}, |J|=l \\ J = \{j_1 < \dots < j_l\}}} ((-1)^{\sum_{i=1}^l n-l+i-j_i}\prod_{k \in \{1,\dots,n\} \setminus J} m_k)_{\epsilon} [\underset{n-1-l \text{ terms}}{\underbrace{-1,\dots,-1}},\overline{u_{j_1}},\dots,\overline{u_{j_l}}]\\
	& + \sum_{p=1}^n \sum_{l=p}^n \sum_{\substack{J \subset \{1,\dots,n\}, |J|=l \\ J = \{j_1 < \dots < j_l\}}} (\sum_{I \subset \{1,\dots,l\}, |I|=p} \eta^p \chi^{\odd}(\prod_{i \in I} m_{j_i} \times \prod_{k \in \{1,\dots,n\} \setminus J} m_k)) [\underset{n-1+p-l \text{ terms}}{\underbrace{-1,\dots,-1}},\overline{u_{j_1}},\dots,\overline{u_{j_l}}] 
	\end{align*}
\end{theorem}

\begin{remark} This last formula may seem daunting, but for $n = 1$ it is merely 
	\[\disp \partial^{\pi}_v([\pi^mu]) = m_{\epsilon} + \eta \chi^ {\odd}(m) [\overline{u}]\]
	(i.e. $\partial^{\pi}_v([\pi^mu]) = \langle \overline{u}\rangle m_{\epsilon}$, similarly to the case $n \leq 0$ where $\partial^{\pi}_v(\langle \pi^mu\rangle \eta^{-n}) = \langle \overline{u}\rangle \eta^{-n+1} m_{\epsilon}$, see Lemma \ref{chilemma}),
	for $n = 2$ it is merely 
	\begin{align*}
	\disp \partial^{\pi}_v([\pi^{m_1}u_1,\pi^{m_2}u_2]) & = (m_1m_2)_{\epsilon} [-1] + (-m_2)_{\epsilon} [\overline{u_1}] + (m_1)_{\epsilon} [\overline{u_2}] \\
	& + \eta \chi^{\odd}(m_1m_2) [-1,\overline{u_1}] + \eta \chi^{\odd}(m_1m_2) [-1,\overline{u_2}] \\ 
	& + (\eta \chi^{\odd}(m_1) + \eta \chi^{\odd}(m_2)) [\overline{u_1},\overline{u_2}] \\
	& + \eta^2 \chi^{\odd}(m_1m_2) [-1,\overline{u_1},\overline{u_2}]
	\end{align*}
	and so on.
\end{remark}

\begin{proof}
	Let $n \leq 0$, $m \in \Z$ and $u \in \Os_v^*$.  
	\begin{align*}
	\partial^{\pi}_v(\langle \pi^{m}u\rangle \eta^{-n}) & = \partial^{\pi}_v((1 + \eta [\pi^{m}u]) \eta^{-n}) & \\ & = \partial^{\pi}_v((1 + \eta ([\pi^{m}] + [u] + \eta [\pi^{m},u])) \eta^{-n}) & \\
	& = \partial^{\pi}_v((1 + \eta m_{\epsilon} [\pi] + \eta [u] + \eta^2 m_{\epsilon} [\pi,u])\eta^{-n}) & \text{ by \cite[Lemma 3.14]{morel}}\\ 
	& = \eta^{-n} \partial^{\pi}_v(1) + \eta^{-n+1} m_{\epsilon}\partial^{\pi}_v([\pi]) \\
	& \ \ \, + \eta^{-n+1} \partial^{\pi}_v([u]) + \eta^{-n+2} m_{\epsilon} \partial^{\pi}_v([\pi,u]) & \text{ by Prop. \ref{propresidue} and Def. \ref{defmorel}}\\ 
	& = \eta^{-n+1}m_{\epsilon} + \eta^{-n+2}m_{\epsilon}[\overline{u}] & \text{ by Def. \ref{defmorel}}\\ 
	& = (\eta^{-n+1}+\eta^{-n+2}[\overline{u}])\chi^{\odd}(m) & \text{ by Lemma \ref{chilemma}}\\
	& = \langle \overline{u}\rangle \eta^{-n+1} \chi^{\odd}(m) &
	\end{align*} 
	
	Let $n \geq 1, m_1,\dots,m_n \in \Z$, $u_1,\dots,u_n \in \Os_v^*$ and $N := \{1,\dots,n\}$.
	\begin{align*}
	[\pi^{m_1} u_1,\dots,\pi^{m_n}u_n] & = \prod_{i=1}^n ([\pi^{m_i}] + [u_i] + \eta [\pi^{m_i},u_i])\\ & = \prod_{i=1}^n ((m_i)_{\epsilon} [\pi] + [u_i] + \eta (m_i)_{\epsilon} [\pi,u_i]) \text{ by \cite[Lemma 3.14]{morel}}
	\end{align*}
	\begin{align*}
	& \text{ Hence } [\pi^{m_1} u_1,\dots,\pi^{m_n}u_n] = \sum_{l=0}^n \sum_{\substack{J \subset \{1,\dots,n\}, |J|=l \\ J = \{j_1 < \dots < j_l\}}} \prod_{k \in N \setminus J} (m_k)_{\epsilon} \times \epsilon^{\sum_{i=1}^l n-l+i-j_i} [\pi,\dots,\pi,u_{j_1},\dots,u_{j_l}] \\ & \phantom{b} + \sum_{p=1}^n \sum_{l=p}^n \sum_{\substack{J \subset \{1,\dots,n\}, |J|=l \\ J = \{j_1 < \dots < j_l\}}} (\sum_{I \subset \{1,\dots,l\}, |I| = p} \eta^p \times \prod_{i \in I} (m_{j_i})_{\epsilon} \times \prod_{k \in N \setminus J} (m_k)_{\epsilon}) [\pi,\dots,\pi,u_{j_1},\dots,u_{j_l}]
	\end{align*}
	
	We obtained this last equality by developing the product and using  \cite[Corollary 3.8]{morel} ($\epsilon$-graded commutativity), as well as the fact that $\eta \epsilon = \eta$.
	
	The index $p$ corresponds to the number of terms coming from an $\eta (m_i)_{\epsilon} [\pi,u_i]$, the index $l$ corresponds to the number of terms coming from a $[u_i]$ or an $\eta (m_i)_{\epsilon} [\pi,u_i]$, the set $J = \{j_1,\dots,j_l\}$ (with $j_1 < \dots < j_l$) corresponds to the indices of the terms coming from a $[u_i]$ or an $\eta (m_i)_{\epsilon} [\pi,u_i]$ and the set $I$ corresponds to the indices of the $j_i$ such that $u_{j_i}$ comes from an $\eta (m_{j_i})_{\epsilon} [\pi,u_{j_i}]$.
	
	By \cite[Lemma 3.7]{morel} and Lemma \ref{epslemma}: 
	\begin{align*}
	& [\pi^{m_1}u_1,\dots,\pi^{m_n}u_n] = \sum_{l=0}^n \sum_{\substack{J \subset \{1,\dots,n\}, |J|=l \\ J = \{j_1 < \dots < j_l\}}} ((-1)^{\sum_{i=1}^l n-l+i-j_i} \prod_{k \in N \setminus J} m_k)_{\epsilon} [\pi,-1,\dots,-1,u_{j_1},\dots,u_{j_l}] + \\ & \sum_{p=1}^n \sum_{l=p}^n \sum_{\substack{J \subset \{1,\dots,n\}, |J|=l \\ J = \{j_1 < \dots < j_l\}}} (\sum_{I \subset \{1,\dots,l\}, |I|=p} \eta^p (\prod_{i \in I}  m_{j_i} \times \prod_{k \in N \setminus J} m_k)_{\epsilon}) [\pi,-1,\dots,-1,u_{j_1},\dots,u_{j_l}] 
	\end{align*}
	
	By Lemma \ref{chilemma} :
	\begin{align*}
	& [\pi^{m_1}u_1,\dots,\pi^{m_n}u_n] = \sum_{l=0}^n \sum_{\substack{J \subset \{1,\dots,n\}, |J|=l \\ J = \{j_1 < \dots < j_l\}}} ((-1)^{\sum_{i=1}^l n-l+i-j_i} \prod_{k \in N \setminus J} m_k)_{\epsilon} [\pi,-1,\dots,-1,u_{j_1},\dots,u_{j_l}] + \\ & \sum_{p=1}^n \sum_{l=p}^n \sum_{\substack{J \subset \{1,\dots,n\}, |J|=l \\ J = \{j_1 < \dots < j_l\}}} (\sum_{I \subset \{1,\dots,l\}, |I|=p} \eta^p \chi^{\odd}(\prod_{i \in I} m_{j_i} \times \prod_{k \in N \setminus J} m_k)) [\pi,-1,\dots,-1,u_{j_1},\dots,u_{j_l}]
	\end{align*}
	
	By Definition \ref{defmorel} and Proposition \ref{propresidue} , $\partial^{\pi}_v([\pi^{m_1}u_1,\dots,\pi^{m_n}u_n])$ is equal to:
	\begin{align*}
	& \sum_{l=0}^{n-1} \sum_{\substack{J \subset \{1,\dots,n\}, |J|=l \\ J = \{j_1 < \dots < j_l\}}} ((-1)^{\sum_{i=1}^l n-l+i-j_i} \prod_{k \in N \setminus J} m_k)_{\epsilon} [-1,\dots,-1,u_{j_1},\dots,u_{j_l}] + \\ & \sum_{p=1}^n \sum_{l=p}^n \sum_{\substack{J \subset \{1,\dots,n\}, |J|=l \\ J = \{j_1 < \dots < j_l\}}} (\sum_{I \subset \{1,\dots,l\}, |I|=p} \eta^p \chi^{\odd}(\prod_{i \in I} m_{j_i} \times \prod_{k \in N \setminus J} m_k)) [-1,\dots,-1,u_{j_1},\dots,u_{j_l}]
	\end{align*}
	
	Note that the term $l = n$ in the first double sum vanishes since  $\partial^{\pi}_v([u_1,\dots,u_n]) = 0$ (by Definition \ref{defmorel}).
\end{proof}

\section{The Rost-Schmid complex and Rost-Schmid groups}\label{secprelim}

In this appendix, we recall notions about the Rost-Schmid complex and its cohomology groups which are used in our paper. See \cite[Section 3.1]{morel} for recollections about Milnor-Witt $K$-theory. Throughout this appendix, $F$ is a perfect field and $X$ is a smooth finite-type $F$-scheme. We denote the usual generators of the Milnor-Witt $K$-theory ring of $F$ by $[a] \in \K{F}{1}$ (with $a \in F^*$) and $\eta \in \K{F}{-1}$ (see \cite[Definition 3.1]{morel}). We denote $\langle a\rangle := 1 + \eta [a] \in \K{F}{0}$.

\subsection{Definitions and first properties}\label{subdeffirstprop}

We are about to give the definition of the Rost-Schmid complex that Morel gave in \cite[Chapter 5]{morel}. Note that an earlier (equivalent) definition of the Rost-Schmid complex was given in \cite{bargemorel}. (The equivalence of these definitions follows from \cite[Theorem 6.4.5]{morelintro}.) Before we define it, recall Definition \ref{deftwMW} (twisted Milnor-Witt $K$-theory) and the following definition.

\begin{definition}[Determinant of a locally free module] The determinant of a locally free $\Os_X$-module $\Vb$ of constant finite rank $r$, denoted $\det(\Vb)$, is its $r$-th exterior power $\Lambda^r(\Vb)$.
\end{definition}

\begin{definition}[Rost-Schmid complex]\label{defRScomp}
	Let $j \in \Z$ and $\Lb$ be an invertible $\Os_X$-module. The Rost-Schmid complex associated to $X,$ $j$ and $\Lb$ is :
	\[\Cm(X,\KMWF{j}{\Lb}) = \bigoplus_{i \in \N_0} \Cm^i(X,\KMWF{j}{\Lb})\]
	with
	\[\Cm^i(X,\KMWF{j}{\Lb}) = \bigoplus_{x \in X^{(i)}} \KS{\kappa(x)}{j-i}{\nu_x \otimes_{\kappa(x)} \Lb_{|x}}\]
	where $X^{(i)}$ is the set of points of codimension $i$ in $X$, $\Lb_{|x} = \Lb_x \otimes_{\Os_{X,x}} \kappa(x)$ and $\nu_x = \det(\normalsheaf{x}{X})$ with $\normalsheaf{x}{X}$ the normal sheaf of $x$ in $X$, i.e. the dual of $\mathfrak{m}_{X,x}/\mathfrak{m}_{X,x}^2$. We denote $\Cm(X,\KF{j}) := \Cm(X,\KMWF{j}{\Os_X})$.
\end{definition}

Recall Definition \ref{twistedcanonicalres} and the following notation (taken from \cite[pp. 121-122]{morel}).

\begin{notation}\label{notmor}
	Let $x \in X$ be such that $\overline{\{x\}}$ is smooth, $y \in \overline{\{x\}}^{(1)}$ and $\Lb$ be an invertible $\Os_X$-module. We denote by 
	\[\partial^x_y : \KS{\kappa(x)}{*}{\nu_x \otimes_{\kappa(x)} \Lb_{|x}} \to \KS{\kappa(y)}{*-1}{\nu_y \otimes_{\kappa(y)} \Lb_{|y}}\]
	the twisted canonical residue morphism associated to the discrete valuation of $\Os_{\overline{\{x\}},y}$.
\end{notation}

If $\overline{\{x\}}$ is not smooth, the morphism $\partial^x_y : \KS{\kappa(x)}{*}{\nu_x \otimes_{\kappa(x)} \Lb_{|x}} \to \KS{\kappa(y)}{*-1}{\nu_y \otimes_{\kappa(y)} \Lb_{|y}}$ is the sum over the points $z$ above $y$ in the normalisation of $\overline{\{x\}}$ of the composition of the adequate twisted canonical residue morphism and of the transfer morphism associated to $y$ and $z$ (see \cite[Subsection 2.1]{fasel} or Feld's article \cite{feld} (take $M = \un{K}^{\MW}$ in Feld's notations) or D\'eglise's notes \cite{deglisecourse}).

\begin{definition}[Differential of the Rost-Schmid complex]
	Let $j \in \Z$ and $\Lb$ be an invertible $\Os_X$-module. The differential of the Rost-Schmid complex associated to $X$, $j$ and $\Lb$ is the morphism $d_{X,j,\Lb} : \Cm^*(X,\KMWF{j}{\Lb}) \to \Cm^{*+1}(X,\KMWF{j}{\Lb})$, denoted $d$ for short, given by: 
	\[d^i(\sum_{x \in X^{(i)}} k_x) = \sum_{x \in X^{(i)}} \sum_{y \in \overline{\{x\}}^{(1)}} \partial^x_y(k_x)\]
\end{definition}

Note that the sum which appears in the above definition is well-defined since, with the same notations as above, for every $k_x$ the number of $y \in \overline{\{x\}}^{(1)}$ such that $\partial^x_y(k_x) \neq 0$ is finite (see \cite[Sections 4 and 7]{feld} (especially axiom FD) or D\'eglise's notes \cite{deglisecourse}). 

By \cite[Theorem 5.31]{morel}, the Rost-Schmid complex is a complex, i.e. for all $i \in \N_0$, $d^{i+1} \circ d^i = 0$, hence we can define the Rost-Schmid groups as follows.

\begin{definition}[Rost-Schmid groups]\label{defRSgroups}
	Let $i, j \in \Z$, $\Lb$ be an invertible $\Os_X$-module. The $i$-th Rost-Schmid group associated to $X$, $j$ and $\mathcal{L}$, denoted by $\CHMW{X}{i}{j}{\Lb}$, is the $i$-th cohomology group of the Rost-Schmid complex $\Cm(X,\KMWF{j}{\Lb})$, i.e.:
	\[\CHMW{X}{i}{j}{\Lb} = \ker(d^i)/\ima(d^{i-1})\]
	where by convention $d^i = 0$ if $i < 0$. We denote $\CH{X}{i}{j} := \CHMW{X}{i}{j}{\Os_X}$.
\end{definition} 

Note that by definition, for all $i \in \N_0$ and $j \in \Z$, $\Cm^i(\Spec(F),\KF{j}) = \K{F}{j}$ if $i = 0$, to $0$ otherwise, hence $\CH{\Spec(F)}{i}{j} = \K{F}{j}$ if $i = 0$, to $0$ otherwise.

\begin{remark}
	Note that by \cite[Theorem 5.47]{morel} Rost-Schmid groups generalize Chow-Witt groups $\widetilde{CH}^i(X)$: if $X$ is a smooth $F$-scheme and $i \in \N_0$ then $\CH{X}{i}{i} = \widetilde{\Chow}^i(X)$.
\end{remark}

Let us now state the property of homotopy invariance of Rost-Schmid groups.

\begin{theorem}[Theorem 5.38 in \cite{morel}]\label{hominv}
	Let $\pi : \A^1_X \to X$ be the projection, $i \in \N_0$ and $j \in \Z$. The induced morphism $\pi^* : \CH{X}{i}{j} \to \CH{\A^1_X}{i}{j}$ is an isomorphism.
\end{theorem}

Note that it follows from this theorem that for all $n,i \in \N_0$ and $j \in \Z$, $\CH{\A^n_F}{i}{j}$ is canonically isomorphic to $\CH{\Spec(F)}{i}{j}$ hence to $\K{F}{j}$ if $i = 0$, to $0$ otherwise.

We now define boundary triples and boundary maps, which were introduced by Feld in \cite{feld} (following what Rost did in \cite{rost}).

\begin{definition}[Boundary triple]
	A boundary triple is a $5$-tuple $(Z,i,X,j,U)$, or abusively a triple $(Z,X,U)$, with $i : Z \to X$ a closed immersion and $j : U \to X$ an open immersion such that the image of $U$ by $j$ is the complement in $X$ of the image of $Z$ by $i$, where $Z,X,U$ are smooth $F$-schemes of pure dimensions. We denote by $d_Z$ and $d_X$ the dimensions of $Z$ and $X$ respectively and by $\nu_Z$ the determinant of the normal sheaf of $Z$ in $X$.
\end{definition}

\begin{remark}\label{rkcaniso}
	Let $(Z,i,X,j,U)$ be a boundary triple. Note that, similarly to \cite[(3.10)]{rost}, for each integer $m$, the complex $\Cm^{\bullet+d_Z-d_X}(Z,\KMWF{m+d_Z-d_X}{\nu_Z})$ is a subcomplex of $\Cm^{\bullet}(X,\KF{m})$ with quotient complex $\Cm^{\bullet}(U,\KF{m})$, and that we have for each integer $n$ a canonical isomorphism
	\[\Cm^{n}(X,\KF{m}) \simeq \Cm^{n+d_Z-d_X}(Z,\KMWF{m+d_Z-d_X}{\nu_Z}) \oplus \Cm^{n}(U,\KF{m})\]
\end{remark}

\begin{notation}\label{notcaniso}
	We denote the projections by $i^* : \Cm^{\bullet}(X,\KF{*}) \to \Cm^{\bullet+d_Z-d_X}(Z,\KMWF{*+d_Z-d_X}{\nu_Z})$ and $j^* : \Cm^{\bullet}(X,\KF{*}) \to \Cm^{\bullet}(U,\KF{*})$ and the inclusions by $i_* : \Cm^{\bullet+d_Z-d_X}(Z,\KMWF{*+d_Z-d_X}{\nu_Z}) \to \Cm^{\bullet}(X,\KF{*})$ and $j_* : \Cm^{\bullet}(U,\KF{*}) \to \Cm^{\bullet}(X,\KF{*})$ (see Remark \ref{rkcaniso}).
\end{notation}

\begin{remark}\label{rkinclsubcomp}
	Note that since $i_*$ (resp. $j^*$) is the inclusion of a subcomplex (resp. the projection to a quotient complex), it commutes with the differentials of the complexes, hence induces a morphism  $i_* : \CHMW{Z}{n}{m}{\nu_Z} \to \CH{X}{{n + d_X - d_Z}}{m + d_X - d_Z}$ (resp. $j^* : \CH{X}{n}{m} \to \CH{U}{n}{m}$). Also note that this morphism $i_*$ coincides with the pushforward along the closed immersion $i : Z \to X$ (see \cite[Subsection 2.3]{fasel}) and that this morphism $j^*$ coincides with the pullback along the open immersion $j : U \to X$ (see \cite[Subsection 2.4]{fasel}).
\end{remark}

\begin{definition}[Boundary map]\label{defboundarymap}
	Let $(Z,i,X,j,U)$ be a boundary triple. The boundary map associated to this boundary triple is the morphism 
	\[\partial : \Cm^{\bullet}(U,\KF{*}) \to \Cm^{\bullet+1+d_Z-d_X}(Z,\KMWF{*+d_Z-d_X}{\nu_Z})\]
	induced by the differential $d$ of the Rost-Schmid complex $\Cm(X,\KF{*})$, i.e.:
	\[\partial = i^* \circ d \circ j_*\]
\end{definition}

The following theorem is a special case of the more general exact triangle theorem in homological algebra (the boundary maps being the connecting morphisms).

\begin{theorem}\label{locseq}
	Let $(Z,i,X,j,U)$ be a boundary triple. The boundary map induces a morphism $\partial : \CH{U}{n+d_X-d_Z}{m+d_X-d_Z} \to \CHMW{Z}{n+1}{m}{\nu_Z}$ and we have the following long exact sequence, called the localization long exact sequence:
	\[\xymatrix{\dots \ar[r] & \CHMW{Z}{n}{m}{\nu_Z} \ar[r]^-{i_*} & \CH{X}{n + d_X - d_Z}{m + d_X - d_Z} \ar[r]^-{j^*} & \ \ \
	\\ 
	\ \ \ \ar[r]^-{j^*} & \CH{U}{n + d_X - d_Z}{m + d_X - d_Z} \ar[r]^-\partial & \CHMW{Z}{n+1}{m}{\nu_Z} \ar[r] & \dots}\]
\end{theorem}

\subsection{The Rost-Schmid groups of punctured affine spaces}\label{subRS}

Let us now compute the Rost-Schmid groups of $\As{n}{F}$ for $n \geq 2$. To do this, we use the following lemma (which is also used in the main part of the paper).

\begin{lemma}\label{propll}
	Let $\Lb$ be an invertible $\Os_X$-module. For all $i,j \in \Z$, the morphism 
	\[\fun{\Cm^i(X,\KMWF{j}{\Lb \otimes \Lb})}{\Cm^i(X,\KF{j})}{\disp \sum_{x \in I} k_x \otimes (l_x \otimes l_x)}{\disp \sum_{x \in I} k_x}\]
	where $I$ is a finite subset of $X^{(i)}$, $k_x \in \KS{\kappa(x)}{j-i}{\nu_x}$ and $l_x \in \Lb_{|x} \setminus \{0\}$, is a well-defined isomorphism which commutes with differentials.
\end{lemma}

\begin{proof}
	Note that elements of $\Cm^i(X,\KMWF{j}{\Lb \otimes \Lb})$ are of the form $\sum_{x \in I} m_x \otimes t_x$ with $I$ a finite subset of $X^{(i)}$, $m_x \in \K{\kappa(x)}{j-i}$ and $t_x \in \Z[(\nu_x \otimes (\Lb \otimes \Lb)_{|x}) \setminus \{0\}]$. Let $x \in I$. Since $\nu_x \otimes (\Lb \otimes \Lb)_{|x}$ is a $\kappa(x)$-vector space of dimension $1$, there exist $n_x \in \K{\kappa(x)}{j-i}$ and $s_x \in (\nu_x \otimes (\Lb \otimes \Lb)_{|x}) \setminus \{0\}$ such that $m_x \otimes t_x = n_x \otimes s_x$. By definition of $\KS{\kappa(x)}{j-i}{\nu_x}$, there exist $h_x \in \KS{\kappa(x)}{j-i}{\nu_x}$ and $l_x,r_x \in \Lb_{|x} \setminus \{0\}$ such that $n_x \otimes s_x = h_x \otimes (l_x \otimes r_x)$. Since $\Lb_{|x}$ is a $\kappa(x)$-vector space of dimension $1$, there exists $v_x \in \kappa(x)^*$ such that $r_x = v_x l_x$. It follows that $h_x \otimes (l_x \otimes r_x) = \langle v_x\rangle h_x \otimes (l_x \otimes l_x)$. Denoting $k_x := \langle v_x\rangle h_x$, we get that elements of $\Cm^i(X,\KMWF{j}{\Lb \otimes \Lb})$ are of the form $\disp \sum_{x \in I} k_x \otimes (l_x \otimes l_x)$ with $I$ a finite subset of $X^{(i)}$, $k_x \in \KS{\kappa(x)}{j-i}{\nu_x}$ and $l_x \in \Lb_{|x} \setminus \{0\}$.
	
	This morphism is well-defined since if $\sum_{x \in I} k_x \otimes (l_x \otimes l_x) = \sum_{x \in J} k'_x \otimes (l'_x \otimes l'_x)$ with $I,J$ finite subsets of $X^{(i)}$, $k_x,k'_x \in \KS{\kappa(x)}{j-i}{\nu_x}$ and $l_x,l'_x \in \Lb_{|x} \setminus \{0\}$, then for all $x \in I \cup J \setminus (I \cap J)$, $k_x = k'_x = 0$, and for all $x \in I \cap J$, $l'_x = u_x l_x$ for some $u_x \in F^*$ and $k'_x \otimes (l'_x \otimes l'_x) = \langle u_x^2\rangle k'_x \otimes (l_x \otimes l_x) = k'_x \otimes (l_x \otimes l_x)$ hence $k'_x \otimes (l_x \otimes l_x) = k_x \otimes (l_x \otimes l_x)$ hence $k'_x = k_x$. The preceding equality $k_x \otimes (l'_x \otimes l'_x) = k_x \otimes (l_x \otimes l_x)$ shows that the morphism
	\[\fun{\Cm^i(X,\KF{j})}{\Cm^i(X,\KMWF{j}{\Lb \otimes \Lb})}{\disp \sum_{x \in I} k_x}{\disp \sum_{x \in I} k_x \otimes (l_x \otimes l_x)}\]
	is well-defined, which shows that the morphism in the statement is an isomorphism. The commutation with differentials is straightforward.
\end{proof}

\begin{definition}\label{defindiso}
	Let $n \geq 2, j \in \Z$ be integers and $o : \det(\normalsheaf{\{0\}}{\A^n_F}) \to \Os_{\{0\}} \otimes \Os_{\{0\}}$ be an isomorphism. The isomorphism $o$ gives rise to an isomorphism $\CHMW{\{0\}}{0}{j}{\det(\normalsheaf{\{0\}}{\A^n_F})}$ $\to \CHMW{\{0\}}{0}{j}{\Os_{\{0\}} \otimes \Os_{\{0\}}}$ hence to an isomorphism $\widetilde{o} : \CHMW{\{0\}}{0}{j}{\det(\normalsheaf{\{0\}}{\A^n_F})}$ $\to \CH{\{0\}}{0}{j} = \K{F}{j}$ by Lemma \ref{propll}. We call $\widetilde{o}$ the isomorphism induced by $o$.
\end{definition}

\begin{proposition}\label{RS}
	Let $n \geq 2,i \geq 0$ and $j \in \Z$ be integers. Denoting by $\psi : \As{n}{F} \to \A^n_F$ the inclusion and by $\partial$ the boundary map associated to the boundary triple $(\{0\},\A^n_F,\As{n}{F})$, the morphisms $\psi^* : \K{F}{j} \simeq \CH{\A^n_F}{0}{j} \to \CH{\As{n}{F}}{0}{j}$ and $\partial : \CH{\As{n}{F}}{n-1}{j} \to \CHMW{\{0\}}{0}{j-n}{\det(\normalsheaf{\{0\}}{\A^n_F})} \simeq \K{F}{j-n}$ are isomorphisms and if $i \notin \{0,n-1\}$ then $\CH{\As{n}{F}}{i}{j} = 0$.
\end{proposition}

\begin{proof}
	The localization long exact sequence (see Theorem \ref{locseq}) associated to the boundary triple $(\{0\},\varphi,\A^n_F,\psi,\As{n}{F})$ gives us the following exact sequences for all $j \in \Z$ and $i \notin \{0,n-1\}$:
	\[\xymatrix{0 \ar[r] & \CH{\A^n_F}{0}{j} \ar[r]^-{\psi^*} & \CH{\As{n}{F}}{0}{j} \ar[r] & 0}\]
	\[\xymatrix{0 \ar[r] & \CH{\As{n}{F}}{n-1}{j} \ar[r]^-{\partial} & \CHMW{\{0\}}{0}{j-n}{\det(\normalsheaf{\{0\}}{\A^n_F})} \ar[r] & 0}\]
	\[\xymatrix{0 \ar[r] & \CH{\As{n}{F}}{i}{j} \ar[r] & 0}\]
\end{proof}

\begin{remark}
	Note that the Rost-Schmid groups of $\As{n}{F}$ are already known (combine \cite[Lemma 4.5]{asokfasel} with \cite[Corollary 5.43]{morel}, \cite[Example 1.5.1.19]{feldphd} and \cite[Theorem 5.46]{morel}), but the explicit definition of isomorphisms we did above is important for the two following definitions, the first of which is used in the definition of the quadratic linking degree (see Definition \ref{defqld}) and the second of which is used in the definition of the ambient quadratic linking degree (see Definition \ref{defaqld}).
\end{remark} 

\begin{definition}[The conventional isomorphism for $\As{2}{F}$]\label{conviso}
	The conventional isomorphism 
	\[\zeta : \CH{\As{2}{F}}{1}{0} \to \W(F)\]
	is the composite of the boundary map 
	\[\partial : \CH{\As{2}{F}}{1}{0} \to \CHMW{\{0\}}{0}{-2}{\det(\normalsheaf{\{0\}}{\A^2_F})}\]
	(which is an isomorphism by Proposition \ref{RS}), of the isomorphism 
	\[\CHMW{\{0\}}{0}{-2}{\det(\normalsheaf{\{0\}}{\A^2_F})} \to \K{F}{-2}\]
	induced by the isomorphism $\det(\normalsheaf{\{0\}}{\A^2_F}) \to \Os_{\{0\}} \otimes \Os_{\{0\}}$ which sends $\overline{u}^* \wedge \overline{v}^*$ to $1 \otimes 1$, where $\A^2_F = \Spec(F[u,v])$ (see Definition \ref{defindiso}) and of the canonical isomorphism (which sends $\eta^2$ to~$1$)
	\[\K{F}{-2} \to \W(F)\]
\end{definition}

\begin{definition}[The conventional isomorphism for $\As{4}{F}$]\label{convisobis}
	The conventional isomorphism 
	\[\zeta' : \CH{\As{4}{F}}{3}{2} \to \W(F)\]
	is the composite of the boundary map 
	\[\partial : \CH{\As{4}{F}}{3}{2} \to \CHMW{\{0\}}{0}{-2}{\det(\normalsheaf{\{0\}}{\A^4_F})}\]
	(which is an isomorphism by Proposition \ref{RS}), of the isomorphism
	\[\CHMW{\{0\}}{0}{-2}{\det(\normalsheaf{\{0\}}{\A^4_F})} \to \K{F}{-2}\]
	induced by the isomorphism $\det(\normalsheaf{\{0\}}{\A^4_F}) \to \Os_{\{0\}} \otimes \Os_{\{0\}}$ which sends $\overline{x}^* \wedge \overline{y}^* \wedge \overline{z}^* \wedge \overline{t}^*$ to $1 \otimes 1$, where $\A^4_F = \Spec(F[x,y,z,t])$ (see Definition \ref{defindiso}) and of the canonical isomorphism (which sends $\eta^2$ to $1$)
	\[\K{F}{-2} \to \W(F)\]
\end{definition}

\subsection{The intersection product of oriented divisors}\label{subprod}

The intersection product is defined from the exterior product (see \cite[Subsection 3.1]{fasel}), which is also known as the cross product (see \cite[Section 11]{feld}) and the pullback along the diagonal (see \cite[Subsection 3.3]{fasel}), which is also known as the Gysin morphism induced by the diagonal (see \cite[Section 10]{feld}).

\begin{definition}[The intersection product]\label{defintprod}
	Let $\Delta : X \to X \times X$ be the diagonal. The intersection product $\cdot : \CH{X}{i}{j} \times \CH{X}{i'}{j'} \to \CH{X}{i+i'}{j+j'}$ is the composite of the exterior product $\mu : \CH{X}{i}{j} \times \CH{X}{i'}{j'} \to \CH{X \times X}{i+i'}{j+j'}$ with the pullback (a.k.a. Gysin morphism) $\Delta^* : \CH{X \times X}{i+i'}{j+j'} \to \CH{X}{i+i'}{j+j'}$. 
\end{definition}

The proposition below states that the intersection product is a product.

\begin{proposition}[Subsection 3.4 in \cite{fasel} or Theorem 11.6 in \cite{feld}]\label{intscalar}
	The intersection product makes $\disp \bigoplus_{i \in \N_0} \widetilde{\Chow}^i(X)$ into a graded $\K{F}{0}$-algebra, which is called the Chow-Witt ring.
\end{proposition}

In this paper, we are interested in the intersection product of oriented divisors and often use the following proposition.

\begin{proposition}[Subsection 3.4 in \cite{fasel}]\label{comintprod}
	Let $c_1,c_2$ be oriented divisors in $X$, i.e. $c_1,c_2 \in \CH{X}{1}{1}$. The intersection product of $c_1$ with $c_2$, denoted by $c_1 \cdot c_2$, is $\langle -1\rangle$-commutative:
	\[c_2 \cdot c_1 = \langle -1\rangle (c_1 \cdot c_2)\]
\end{proposition}

The following formula (or rather the formula in Corollary \ref{intformula}) is used to compute the intersection product in Sections \ref{sechowto} and \ref{seccompute}. This theorem, which has been proved by D\'eglise, will be made available in the second part of his notes \cite{deglisecourse}; we give a proof sketch below.

\begin{theorem}\label{intformuladeg}
	Let $X$ be a smooth $F$-scheme. Let $D_1,D_2$ be smooth integral divisors in $X$ such that $D_1 \cap D_2$ is of codimension $2$ in $X$. For all $i \in \{1,2\}$, let $g_i$ be a local parameter for $D_i$, i.e. $g_i$ is a uniformizing parameter for $\Os_{X,D_i}$. The intersection product of $1 \otimes \overline{g_1}^* \in \CH{X}{1}{1}$ (over the generic point of $D_1$) with $1 \otimes \overline{g_2}^* \in \CH{X}{1}{1}$ (over the generic point of $D_2$) is the class in $\CH{X}{2}{2}$ of the sum over the generic points $x$ of the irreducible components of $D_1 \cap D_2$ of $(m_x)_\epsilon \langle u_x\rangle \otimes (\overline{\pi_x}^* \otimes \overline{g_1}^*)$ (over the point $x$), where $\pi_x$ is a uniformizing parameter for $\Os_{X,x}/(g_1)$, $u_x$ is a unit in $\Os_{X,x}/(g_1)$ and $m_x \in \Z$ such that $g_2 = u_x \pi_x^{m_x} \in \Os_{X,x}/(g_1)$.
\end{theorem}

The ideas of the proof are the following:
\begin{itemize}
	\item Reduce the problem to the case where $D_1 = \di(g_1)$. 
	\item Denoting by $i_1 : D_1 \to X$ the inclusion and by $\Theta_1 : \CH{D_1}{0}{0} \to \CHMW{D_1}{0}{0}{\nu_{D_1}}$ (where $\nu_{D_1}$ is the determinant of the normal sheaf of $D_1$ in $X$) the isomorphism which sends $1$ to $1 \otimes \overline{g_1}^*$, check that $1 \otimes \overline{g_1}^* \in \CH{X}{1}{1}$ is equal to $(i_1)_*(\Theta_1(1))$.
	\item Use the projection formula (Theorem 3.19 in \cite{fasel}) to show that:
	\[(i_1)_*(\Theta_1(1)) \cdot (1 \otimes \overline{g_2}^*) = (i_1)_*(\Theta_1(1) \cdot (i_1)^*(1 \otimes \overline{g_2}^*))\]
	\item Use Proposition 3.2.15 in \cite{deglisefeldjin}, which states that if $i$ is the closed immersion of a principal divisor $D = \di(\pi)$ and $j$ is the complementary open immersion to $i$, then $i^! = \partial \circ \gamma_{[\pi]} \circ j^!$, to show that $(i_1)^* = \partial_1 \circ \gamma_{[g_1]} \circ (j_1)^*$, with $j_1$ the complementary open immersion to $i_1$, $\partial_1$ the boundary map associated to the boundary triple $(D_1,i_1,X,j_1,X \setminus D_1)$ and $\gamma_{[g_1]}$ the multiplication by $[g_1]$.
	\item Deduce from the previous steps the equality below and conclude:
	\[(1 \otimes \overline{g_1}^*) \cdot (1 \otimes \overline{g_2}^*) = (i_1)_*(\Theta_1(1) \cdot (\partial_1 \circ \gamma_{[g_1]} \circ (j_1)^*)(1 \otimes \overline{g_2}^*))\]
\end{itemize}


\begin{corollary}\label{intformula}
	Let $X$ be a smooth $F$-scheme. Let $D_1,D_2$ be smooth integral divisors in $X$ such that $D_1 \cap D_2$ is of codimension $2$ in $X$. For all $i \in \{1,2\}$, let $g_i$ be a local parameter for $D_i$ and $f_i$ be a unit in $\kappa(D_i) = \Os_{X,D_i} / \mathfrak{m}_{X,D_i}$ such that for all generic points $x$ of irreducible components of $D_1 \cap D_2$, $f_i \in \kappa(x) = \Os_{X,x} / \mathfrak{m}_{X,x}$ is a unit. The intersection product of $\langle f_1\rangle \otimes \overline{g_1}^* \in \CH{X}{1}{1}$ (over the generic point of $D_1$) with $\langle f_2\rangle \otimes \overline{g_2}^* \in \CH{X}{1}{1}$ (over the generic point of $D_2$) is the class in $\CH{X}{2}{2}$ of the sum over the generic points $x$ of the irreducible components of $D_1 \cap D_2$ of $(m_x)_\epsilon \langle f_1 f_2 u_x\rangle \otimes (\overline{\pi_x}^* \otimes \overline{g_1}^*)$ (over the point $x$), where $\pi_x$ is a uniformizing parameter for $\Os_{X,x}/(g_1)$, $u_x$ is a unit in $\Os_{X,x}/(g_1)$ and $m_x \in \Z$ such that $g_2 = u_x \pi_x^{m_x} \in \Os_{X,x}/(g_1)$.
\end{corollary}

\begin{proof}
	First note that, with the notations above, $f_i \in \kappa(x)$ is well-defined since if $f_i$ and $f'_i$ are two representatives in $\Os_{X,D_i}$ of $f_i \in \kappa(D_i)$ (hence differ by an element of $\mathfrak{m}_{X,D_i}$) and if $f_i, f'_i \in \Os_{X,x}$ are sent by the canonical morphism $\psi : \Os_{X,x} \to \Os_{X,D_i}$ to $f_i, f'_i \in \Os_{X,D_i}$ respectively, then $f_i,f'_i \in \Os_{X,x}$ differ by an element of $\mathfrak{m}_{X,x}$ (since $\psi^{-1}(\mathfrak{m}_{X,D_i}) \subset \mathfrak{m}_{X,x}$).
	
	Note that for all $i \in \{1,2\}$, $\langle f_i\rangle \otimes \overline{g_i}^* = 1 \otimes \overline{f_ig_i}^*$ with $f_i g_i$ a local parameter for $D_i$ ($\overline{f_ig_i} \in \mathfrak{m}_{X,D_i}/\mathfrak{m}_{X,D_i}^2$ is well-defined since $f_i \in \Os_{X,D_i} / \mathfrak{m}_{X,D_i}$ and $g_i \in \mathfrak{m}_{X,D_i}$ and (a representative of) $f_i g_i \in \mathfrak{m}_{X,D_i}$ is a generator of $\mathfrak{m}_{X,D_i}$ since (a representative of) $f_i$ is a unit in $\Os_{X,D_i}$ and $g_i$ is a generator of $\mathfrak{m}_{X,D_i}$).
	
	Therefore, by Theorem \ref{intformuladeg}, the intersection product of $\langle f_1\rangle \otimes \overline{g_1}^*$ with $\langle f_2\rangle \otimes \overline{g_2}^*$ is the sum over the generic points $x$ of the irreducible components of $D_1 \cap D_2$ of $(m_x)_\epsilon \langle v_x\rangle \otimes (\overline{\pi_x}^* \otimes \overline{f_1g_1}^*)$ (over the point $x$), where $\pi_x$ is a uniformizing parameter for $\Os_{X,x}/(f_1g_1)$, $v_x$ is a unit in $\Os_{X,x}/(f_1g_1)$ and $m_x \in \Z$ such that $f_2g_2 = v_x \pi_x^{m_x} \in \Os_{X,x}/(f_1g_1)$.
	
	Note that since $f_2$ is a unit in $\kappa(x)$, $u_x := f_2^{-1}v_x$ is a unit in $\kappa(x)$ and $(m_x)_\epsilon \langle v_x\rangle \otimes (\overline{\pi_x}^* \otimes \overline{f_1g_1}^*) = (m_x)_\epsilon \langle f_2u_x\rangle \otimes (\overline{\pi_x}^* \otimes \overline{f_1g_1}^*)$. Further note that since $f_1$ is a unit in $\kappa(x)$, the ideal $(f_1g_1)$ is equal to the ideal $(g_1)$ in $\Os_{X,x}$ and $(m_x)_\epsilon \langle f_2u_x\rangle \otimes (\overline{\pi_x}^* \otimes \overline{f_1g_1}^*) = (m_x)_\epsilon \langle f_1f_2u_x\rangle \otimes (\overline{\pi_x}^* \otimes \overline{g_1}^*)$. Finally note that, by definition of $u_x$, $g_2 = u_x \pi_x^{m_x} \in \Os_{X,x}/(g_1)$.
\end{proof}

\bibliographystyle{alpha}
\bibliography{Bibliography}

\end{document}